\documentclass[11pt]{amsart}
\usepackage{amssymb}
\usepackage{verbatim}
\usepackage{hyperref}
\usepackage{color}

\begin{document}
\setcounter{tocdepth}{1}

\newtheorem{theorem}{Theorem}    
\newtheorem{proposition}[theorem]{Proposition}
\newtheorem{conjecture}[theorem]{Conjecture}
\def\theconjecture{\unskip}
\newtheorem{corollary}[theorem]{Corollary}
\newtheorem{lemma}[theorem]{Lemma}
\newtheorem{sublemma}[theorem]{Sublemma}
\newtheorem{fact}[theorem]{Fact}
\newtheorem{observation}[theorem]{Observation}
\theoremstyle{definition}
\newtheorem{definition}{Definition}
\newtheorem{notation}[definition]{Notation}
\newtheorem{remark}[definition]{Remark}
\newtheorem{question}[definition]{Question}
\newtheorem{questions}[definition]{Questions}

\newtheorem{example}[definition]{Example}
\newtheorem{problem}[definition]{Problem}
\newtheorem{exercise}[definition]{Exercise}

\numberwithin{theorem}{section}
\numberwithin{definition}{section}
\numberwithin{equation}{section}

\def\reals{{\mathbb R}}
\def\torus{{\mathbb T}}
\def\integers{{\mathbb Z}}
\def\rationals{{\mathbb Q}}
\def\naturals{{\mathbb N}}
\def\complex{{\mathbb C}\/}
\def\heis{{\mathbb H}\/}
\def\distance{\operatorname{distance}\,}
\def\diststar{{\rm dist}^\star}
\def\sym{\operatorname{Symm}\,}
\def\support{\operatorname{support}\,}
\def\dist{\operatorname{dist}}
\def\Span{\operatorname{span}\,}
\def\degree{\operatorname{degree}\,}
\def\kernel{\operatorname{kernel}\,}
\def\dim{\operatorname{dim}\,}
\def\codim{\operatorname{codim}}
\def\trace{\operatorname{trace\,}}
\def\Span{\operatorname{span}\,}
\def\dimension{\operatorname{dimension}\,}
\def\codimension{\operatorname{codimension}\,}
\def\Gl{\operatorname{Gl}}
\def\nullspace{\scriptk}
\def\kernel{\operatorname{Ker}}
\def\ZZ{ {\mathbb Z} }
\def\p{\partial}
\def\rp{{ ^{-1} }}
\def\Re{\operatorname{Re} }
\def\Im{\operatorname{Im} }
\def\ov{\overline}
\def\eps{\varepsilon}
\def\lt{L^2}
\def\diver{\operatorname{div}}
\def\curl{\operatorname{curl}}
\def\etta{\eta}
\newcommand{\norm}[1]{ \|  #1 \|}
\def\expect{\mathbb E}
\def\bull{$\bullet$\ }
\def\det{\operatorname{det}}
\def\Det{\operatorname{Det}}
\def\multiR{\mathbf R}
\def\bestA{\mathbf A}
\def\bestB{\mathbf B}
\def\bestC{\mathbf C}
\def\Apq{\mathbf A_{p,q}}
\def\Apqr{\mathbf A_{p,q,r}}
\def\rank{\operatorname{rank}}
\def\rankk{\mathbf r}
\def\diameter{\operatorname{diameter}}
\def\bp{\mathbf p}
\def\bq{\mathbf q}
\def\bff{\mathbf f}
\def\bg{\mathbf g}
\def\bh{\mathbf h}
\def\bv{\mathbf v}
\def\bw{\mathbf w}
\def\bu{\mathbf u}
\def\bx{\mathbf x}
\def\by{\mathbf y}
\def\bz{\mathbf z}
\def\bzero{\mathbf 0}
\def\bL{\mathbf L}
\def\essd{\operatorname{essential\ diameter}}
\def\gl{\operatorname{Gl}}
\def\kbe{{\scriptk}_\be}

\def\mab{M}
\def\t2{\tfrac12}

\newcommand{\abr}[1]{ \langle  #1 \rangle}
\def\unitQ{{\mathbf Q}}
\def\mbfp{{\mathbf P}}

\def\aff{\operatorname{Aff}}
\def\T{{\mathcal T}}

\def\repair{\medskip \hrule \hrule \medskip}

\def\ovl{\overline{L}}
\def\bard{\bar\delta}
\def\tdelt{\tilde\delta}
\def\essinf{\operatorname{essinf}}
\def\esssup{\operatorname{esssup}}

\newcommand{\Norm}[1]{ \Big\|  #1 \Big\| }
\newcommand{\set}[1]{ \left\{ #1 \right\} }
\newcommand{\sset}[1]{ \{ #1 \} }

\def\one{{\mathbf 1}}
\def\onei{{\mathbf 1}_I}
\def\onee{{\mathbf 1}_E}
\def\onea{{\mathbf 1}_A}
\def\oneb{{\mathbf 1}_B}
\def\wonee{\widehat{\mathbf 1}_E}
\newcommand{\modulo}[2]{[#1]_{#2}}

\def\rint{ \int_{\reals^+} }
\def\Abest{{\mathbb A}}
\def\op{\operatorname{Op}}
\def\and{ \ \text{ and }\ }

\def\barrier{\bigskip\hrule\hrule\bigskip}

\def\symdif{\,\Delta\,}
\def\akd{{\mathbf A}_{k,d}}
\def\ak{{\mathbf A}_{k}}

\def\scriptf{{\mathcal F}}
\def\scripts{{\mathcal S}}
\def\scriptq{{\mathcal Q}}
\def\scriptg{{\mathcal G}}
\def\scriptm{{\mathcal M}}
\def\scriptb{{\mathcal B}}
\def\scriptc{{\mathcal C}}
\def\scriptt{\Phi}
\def\scripti{{\mathcal I}}
\def\scripte{{\mathcal E}}
\def\scriptv{{\mathcal V}}
\def\scriptw{{\mathcal W}}
\def\scriptu{{\mathcal U}}
\def\scripta{{\mathcal A}}
\def\scriptr{{\mathcal R}}
\def\scripto{{\mathcal O}}
\def\scripth{{\mathcal H}}
\def\scriptd{{\mathcal D}}
\def\scriptl{{\mathcal L}}
\def\scriptn{{\mathcal N}}
\def\scriptp{{\mathcal P}}
\def\scriptk{{\mathcal K}}
\def\scriptP{{\mathcal P}}
\def\scriptj{{\mathcal J}}
\def\scriptz{{\mathcal Z}}
\def\frakv{{\mathfrak V}}
\def\frakG{{\mathfrak G}}
\def\frakA{{\mathfrak A}}
\def\frakB{{\mathfrak B}}
\def\frakC{{\mathfrak C}}
\def\frakf{{\mathfrak F}}
\def\fraki{{\mathfrak I}}
\def\fcross{{\mathfrak F^{\times}}}

\def\boldf{\mathbf f}
\def\bolda{\mathbf a}
\def\boldb{\mathbf b}
\def\boldg{\mathbf g}
\def\bF{\mathbf F}
\def\bE{\mathbf E}
\def\bI{\mathbf I}
\def\be{\mathbf e}
\def\bA{\mathbf A}
\def\ba{\mathbf a}
\def\bB{\mathbf b}
\def\bs{\mathbf s}
\def\Estar{E^\star}
\def\bEstar{\mathbf E^\star}
\def\bEdagger{\mathbf E^\dagger}
\def\bAstar{\mathbf A^\star}
\def\Ee{{\mathcal E}}
\def\dorbit{\dist(\bE,\scripto(\bEstar))}

\def\tf{\tfrac{1}{2}}
\def\fg{\frakG}

\author{Michael Christ}
\address{
        Michael Christ\\
        Department of Mathematics\\
        University of California \\
        Berkeley, CA 94720-3840, USA}
\email{mchrist@berkeley.edu}
\thanks{Research supported in part by NSF grant DMS-1363324.}

\date{February 22, 2017. Revised March 23, 2017.}

\title
[Equality in Brascamp-Lieb-Luttinger inequalities]
{Equality in Brascamp-Lieb-Luttinger Inequalities}

\begin{abstract} 
An inequality of Brascamp-Lieb-Luttinger generalizes the Riesz-Sobolev
inequality, stating that certain multilinear functionals,
acting on nonnegative functions of one real variable  with
prescribed distribution functions, are maximized
when these functions are symmetrized. It is shown that under certain hypotheses,
when the functions are indicator functions of sets of prescribed measures, 
then up to the natural translation symmetries
of the inequality, the maximum is attained only by
intervals centered at the origin. 
Moreover, a quantitative form of this
uniqueness is established, sharpening the inequality. 
The hypotheses include an auxiliary genericity assumption which may not be necessary.
\end{abstract}

\maketitle

\section{Introduction}

For any Lebesgue measurable set $E\subset\reals^d$ satisfying
$0<|E|<\infty$, define $E^\star\subset\reals^d$ to be the closed ball
centered at $0$ satisfying $|E^\star|=|E|$. 
Let $J$ be a finite index set, and let $m\in\naturals$. 
Let $\scriptl=\{L_j: j\in J\}$ be a finite family of surjective
linear mappings $L_j:\reals^m\to\reals^d$.
Let $\bff = (f_j: j\in J)$ where $f_j:\reals^d\to[0,\infty]$ are Lebesgue measurable.
Define 
\begin{equation}
\Phi_\scriptl(\bff) = \int_{\reals^m} \prod_{j\in J} f_j\circ L_j,
\end{equation}
integrating with respect to Lebesgue measure.
In this paper, we analyze maximizers of $\Phi_\scriptl$
among all tuples of  indicator functions of sets of specified Lebesgue measures, in
the foundational case in which the dimension of the target spaces $\reals^d$ equals $1$.
By a maximizer, we will always mean a maximizer among tuples having specified measures.

Write $\bE = (E_j: j\in J)$
and $\bE^\star = (E_1^\star: j\in J)$.
Write $\scriptt_\scriptl(\bE) = \scriptt_\scriptl(\one_{E_j}: j\in J)$.
Brascamp, Lieb, and Luttinger \cite{BLL} have proved that for $d=1$,
\begin{equation}\label{eq:BLL} \scriptt_\scriptl(\bE) \le \scriptt_\scriptl(\bE^\star).  \end{equation}
Thus among $n$-tuples of sets with prescribed measures, the configuration in which each set
is an interval centered at the origin is a maximizer of $\scriptt_\scriptl$. 

In what circumstances, and to what degree, are maximizing $n$-tuples of sets unique?
This paper provides an answer, under circumstances that are rather general,
though not quite maximally so.
Putting it inexactly, we characterize maximizers for data 
in the {\em interior} of the set of all data for which a meaningful characterization may be possible,
under an auxiliary (concrete) genericity hypothesis on the data.
Moreover, we prove uniqueness in a stronger quantitative form, which
is not valid in general on the boundary of the set of such data.

The most fundamental example is the Riesz-Sobolev inequality.  Define 
\begin{equation}
\begin{aligned}
\scriptt_{\text{RS}}(E_1,E_2,E_3) 
& = \iint_{\reals^{d}\times\reals^d}
\one_{E_1}(x)\one_{E_2}(y)\one_{E_3}(-x-y)\,dx\,dy
\\&
= \iint_{\Sigma} \prod_{j=1}^3 \one_{E_j}(x_j)\,d\lambda(x_1,x_2,x_3)
\end{aligned}
\end{equation}
where $\Sigma = \{(x_1,x_2,x_3)\in (\reals^d)^3: x_1+x_2+x_3=0\}$, and $\lambda$ 
is the natural $2d$--dimensional Lebesgue measure on $\Sigma$;
$d\lambda = dx_i\,dx_j$ for any $i\ne j\in\{1,2,3\}$.
If $0<|E_j|<\infty$ for each index $j$,
and if the $3$--tuple of Lebesgue measures $(|E_j|^{1/d}: 1\le j\le 3)$
is strictly admissible in the sense that $|E_k|^{1/d} < |E_i|^{1/d}+|E_j|^{1/d}$
for each permutation $(i,j,k)$ of $(1,2,3)$,
then as was shown by Burchard \cite{burchard},  equality holds if and only if the sets $E_j$ are
(up to Lebesgue null sets) homothetic ellipsoids whose centers $c_j$ satisfy $c_1+c_2+c_3=0$.
In the borderline admissible case in which 
$|E_k|^{1/d} \le |E_i|^{1/d}+|E_j|^{1/d}$ for all permutations
with equality for some permutation, 
$\Phi_{\text{RS}}(E) =\Phi_{\text{RS}}(E^\star)$
if and only if $E_i,E_j,E_k$ are homothetic convex sets satisfying $-E_k = E_i + E_j$.


A trivial necessary and sufficient condition for 
$\Phi_{\text{RS}}(E)$ to be equal to $\Phi_{\text{RS}}(E^\star)$ is that
$-E_k$ should contain the sumset $E_i+E_j$, except for a Lebesgue null set.
$|E_i+E_j|^{1/d}$ can in general be as small as $|E_i|^{1/d}+|E_j|^{1/d}$,
so if $|E_k|^{1/d}>|E_i|^{1/d}+|E_j|^{1/d}$ then 
no conclusion can be drawn from equality except that
$-E_k$ contains $E_i+E_j$ up to a null set;
$-E_k\setminus (E_i+E_j)$ is an arbitrary subset of
$\reals^d\setminus (E_i+E_j)$ of measure $|E_k|-|E_i+E_j|$.
The nonadmissible case is in this sense degenerate, and will not be discussed in this paper. 


A second example is that of the Gowers forms $\scriptt_k$, for $2\le k\in\naturals$,
defined by
\begin{equation}
\scriptt_k(\bff) = \iint_{\reals^d\times (\reals^d)^{k}} 
\prod_{\alpha\in\{0,1\}^k} f_\alpha(x+\alpha\cdot\bh)\,dx\,d\bh
\end{equation}
where $(x,\bh)\in\reals^d\times(\reals^d)^k$,
$\bff=(f_\alpha: \alpha\in\{0,1\}^k)$, 
and $f_\alpha:\reals^d\to[0,\infty]$.
For Gowers {\em norms}, with $f_\alpha = \one_{E_\alpha}=\one_E$ for every index $\alpha$, 
the conclusions of our main theorems were established in \cite{christgowersnorms}.

The additive Euclidean group $\reals^m$ acts as a group of symmetries of the form $\scriptt_\scriptl$. 
For $y\in\reals^d$ and $f:\reals^d\to[0,\infty]$ define $\tau_y f(x)=f(x+y)$.
For any $\bv\in\reals^m$,
\[ \scriptt_\scriptl(\bff) = \int_{\reals^m} \prod_{j\in J} f_j(L_j(\bx+\bv))\,d\bx
= \int_{\reals^m} \prod_{j\in J} g_j(L_j(\bx))\,d\bx
= \scriptt_\scriptl(\bg)\]
where $g_j = \tau_{L_j(\bv)}f_j$.
Consequently, maximizers $\bE$ are never unique.

Other group actions are present, and are relevant to our discussion. 
The general linear group $\gl(m)$ acts on families $\scriptl$ of linear mappings $L_j:\reals^m\to\reals^d$ 
by $(L_j: j\in J) \mapsto (L_j\circ\psi: j\in J)$. 
The product of the groups of all Lebesgue measure-preserving affine automorphisms of $\reals^d$ acts,
by $(f_j: j\in J)\mapsto (f_j\circ\psi_j: j\in J)$.
The product $(0,\infty)^J$ of copies of the multiplicative group $(0,\infty)$
acts by $((L_j,e_j): j\in J)\mapsto ((r_jL_j,r_j^d e_j): j\in J)$.

There are other possible sources of nonuniqueness, besides the $\reals^m$ translation action.
Suppose for instance that $J=\{1,2,\dots,n\}$, that  $L_j$ is independent of $x_m$ for all $j<n$,
and $L_n(x)$ depends only on $x_m$.
Then $\scriptt_\scriptl(\bE)$ takes the form $|E_n|\tilde\scriptt_\scriptl(E_1,\dots,E_{n-1})$
where $\tilde\scriptt_\scriptl$ is another form of the same general type as $\scriptt$.
Thus $\Phi_\scriptl(\bE)$ depends only on $|E_n|$ and on $(E_1,\dots,E_{n-1})$. 
Our theorems include a nondegeneracy hypothesis which excludes examples like this one;
see condition (iii) of Definition~\ref{defn:nondegenerate}. 

Generalization to families of linear mappings $L_j:\reals^m\to\reals^d$ for $d>1$,
or even $\reals^m\to\reals^{d_j}$ with $d_j$ dependent on $j$, is not addressed in this paper.
The inequality of Brascamp-Lieb-Luttinger does have an extension to higher dimensions \cite{BLL}
with $d_j=d$ for all $j$,
under a symmetry hypothesis involving an action of the product $O(d)^J$
of $d$--dimensional orthogonal groups and an appropriate commutation relation
for $L_j$ in terms of this action. The inverse theorem for the Riesz-Sobolev inequality
was proved \cite{burchard} in two steps, with a first step for $d=1$ exploiting its ordering,
and a second step for higher dimensions which combined the one-dimensional
result with other ingredients.
We hope to extend Theorems~\ref{thm:main} and \ref{thm:main2}
to $d>1$ in the same spirit in a subsequent work,
by combining Theorem~\ref{thm:main2} with the techniques  
used in the analysis of the Riesz-Sobolev inequality in \cite{christRSult}.

A related class of inequalities is the H\"older-Brascamp-Lieb class,
of the form $\Phi_\scriptl(\bff)\le B(\scriptl,\bp)\prod_{j\in J}\norm{f_j}_{L^{p_j}}$,
where $L_j:\reals^d\to\reals^{d_j}$ are surjective linear mappings,
with the dimensions $d_j\ge 1$ of the target spaces arbitrary.  The natural 
analogue
$\Phi_\scriptl(\bff)\le\Phi_\scriptl(\bff^{\bf \star})$
of the  Brascamp-Lieb-Luttinger inequality is not true, in general, in this
level of generality.
In particular, it fails to hold for the Loomis-Whitney inequality for $\reals^d$. 

\section{Definitions and hypotheses}

Several definitions must be introduced before our main results can be formulated.
We specialize for the remainder of the paper to the case $d=1$.

\begin{definition}
Let $\be\in(0,\infty)^J$ and $\scriptl=(L_j: j\in J)$.
$\scriptk_\be\subset\reals^m$ is the closed convex set 
\begin{equation}
\scriptk_\be = \set{x\in\reals^m: |L_j(x)|\le\tfrac12 e_j\ \text{ for each $j\in J$} }.
\end{equation}
\end{definition}

\begin{definition} \label{defn:K_j}
For $j\in J$, $K_j = K_{j,\be,\scriptl}$ is the function from
$\reals^1$ to $[0,\infty]$ defined  by
\begin{equation} \label{eq:Kjdefn}
\int_{A} K_j = \scriptt_\scriptl(\bE) \end{equation}
for every Lebesgue measurable set $A\subset\reals^d$,
where $E_j=A$ and for every $i\in J\setminus\{j\}$, 
$E_i$ is the closed ball centered at $0\in\reals^1$ of measure $e_i$.
\end{definition}
The parameter $e_j$ does not enter into the definition of $K_j$.

Certain properties of these kernels $K_j$ will be exploited in the analysis.
Under the hypothesis that the intersection over $i\in\scriptl$
of the nullspaces of $L_i$ is equal to $\{0\}$, each $K_j$ is finite-valued and continuous. 
Indeed, up to a positive constant factor, each $K_j$ is the $m-1$--dimensional
Lebesgue measure of an $m-1$--dimensional slice of a balanced convex body in $\reals^m$,
and moreover, these slices have finite measures.
$K_j$ is even, and $[0,\infty)\owns r \mapsto K_j(rx)$ is nonincreasing for each $x\in\reals^1$.
By the Brunn-Minkowski inequality, $\log K_j$ is concave in the region in $\reals^1$
in which $K_j$ is strictly positive.
Since $K_j$ is also even and nonnegative, its restriction to $[0,\infty)$ is a nonincreasing function.
Moreover, the one-sided derivatives 
$D^\pm K_j(x) = \lim_{h\to 0^\pm} h^{-1} (K_j(x+h)-K_j(x))$
exist and are finite and nonpositive whenever $x>0$ and $K_j(x)>0$.
If  $0<x<x'$, if $K(x)>0$, and if $D^+K_j(x)<0$, 
then $D^\pm K_j(x')<0$, and $K(x)$ is strictly greater than $K(x')$.

\begin{definition} \label{defn:nondegenerate}
A family
$\scriptl=\{L_j: j\in J\}$ of linear mappings
$L_j:\reals^m\to \reals^1$ is nondegenerate if
\newline
(i) Each $L_j$ is surjective, 
\newline (ii)
For any $i\ne j\in J$, $L_i$ is not a scalar multiple of $L_j$,
and
\newline (iii) For each $j\in J$, $\cap_{i\ne j\in J} \kernel(L_i)=\{0\}$.
\end{definition}

\begin{definition} \label{defn:admissible}
$(\scriptl,\be)$ is admissible if for each $k\in J$
there exists $\bx\in \scriptk_\be$ satisfying $|L_k(\bx)| = e_k/2$.
\end{definition}

\begin{definition} \label{defn:strictlyadmissible}
Let $d=1$, and let $m\ge 2$.
Let $J$ be a finite index set.
Let $\scriptl=(L_j: j\in J)$ be a nondegenerate $J$-tuple of 
linear mappings $L_j:\reals^m\to\reals^1$.
Let $\be = (e_j: j\in J)\in(0,\infty)^J$.
Then $(\scriptl,\be)$ is strictly admissible  if 
the following two conditions hold for each $j\in J$.

(i) 
There exists $\bx\in\scriptk_\be$ satisfying $|L_j(\bx)| = \tfrac12 \be_j$
and $|L_i(\bx)|< \tfrac12\be_i$ for all $j\ne i\in J$.

(ii)
$D^- K_j(e_j/2)$ is strictly negative.
\end{definition}

Condition (i) implies that $K_j(e_j/2)$ is strictly positive.

For $\scriptt_{\text{RS}}$, 
Definition~\ref{defn:strictlyadmissible} of strict admissibility is equivalent
to Burchard's definition \cite{burchard} of this concept, while condition (ii) is redundant.

When the conditions in Definitions~\ref{defn:nondegenerate} and \ref{defn:strictlyadmissible} are satisfied,
$\scriptk_\be$ is a compact convex subset of $\reals^m$, has finitely
many extreme points, and is equal to their convex hull.
For each extreme point $\bx$, there must exist
at least $m$ indices $k\in J$ for which $|L_k(\bx)| = \tfrac12 e_k$.
Moreover, $\{L_j\in\scriptl: |L_j(\bx)|=e_j/2\}$ must span the dual space $\reals^{m*}$ of $\reals^m$.

The next concept will be a hypothesis of our main theorems.

\begin{definition}
Let $m\ge 2$.
Let $\scriptl$ be a nondegenerate finite family of surjective linear mappings $L_j:\reals^m\to\reals^1$.
Let $\be\in (0,\infty)^J$. An extreme point $\bx$ of $\scriptk_\be$
is said to be generic if there exist exactly $m$ indices $k\in J$
for which $|L_k(\bx)| = e_k/2$.

$(\scriptl,\be)$ is said to be generic if every extreme point of $\scriptk_\be$
is generic.
\end{definition}

If $m=2$ and $(\scriptl,\be)$ is nondegenerate and strictly admissible
then $\scriptk_\be$ is necessarily generic. 

The following consequence of genericity will be exploited. 

\begin{lemma} \label{lemma:independence}
Let $m\ge 2$, and let $d=1$.  Let $\scriptl$ be nondegenerate, and let $(\scriptl,\be)$ be generic.
If $J'\subset J$ has cardinality $|J'|\le m$, 
and if there exists $\bx\in \scriptk_\be$ satisfying $|L_j(\bx)|=e_j/2$ for each $j\in J'$, 
then $\scriptl'=\{L_j: j\in J'\}$  is linearly independent. 
\end{lemma}

\begin{proof}
Let $\scriptl',\bx$ satisfy the hypotheses.
If $\scriptl'$ is not linearly independent then consider $S=\{\by\in\scriptk_\be:
L_j(\by)=L_j(\bx) \text{ for every } j\in J'\}$.
This is a compact convex subset of $\reals^m$, so has extreme points. Let $\bz$ be
any extreme point of $S$, and consider $J''=\{j\in J: |L_j(\bz)| = e_j/2\}$.
Then $J''\supset J'$. By the genericity hypothesis, $\{L_j: j\in J''\}$ is linearly independent. 
Therefore the same holds for the subset $J'$.
\end{proof}

Let $\scripto(\bEstar)$ denote the orbit of $\bEstar$
under the translation symmetry group $\reals^m$. 
The natural notion of distance from $\bE$ to $\scripto(\bEstar)$ 
is as follows.

\begin{definition}
\begin{equation} \label{eq:distancedefn}
\dist(\bE,\scripto(\bE^\star)) = \inf_{\bv\in\reals^m} \max_{j\in J} |E_j\symdif (E_j^\star + L_j(\bv))|.
\end{equation}
\end{definition}

It is elementary that for each tuple $\bE$ of sets with finite, positive measures, 
this infimum is actually attained by some $\bv$.

\section{Main results}

\begin{theorem} \label{thm:main}
Let $d=1$ and $m\ge 2$.
Let $\scriptl=\{L_j: j\in J\}$ be a nondegenerate finite collection of 
linear mappings $L_j:\reals^m\to\reals^1$.
Let $\be\in(0,\infty)^J$. Suppose that $(\scriptl,\be)$ is strictly admissible and generic.
Let $\bE$ be a $J$-tuple of Lebesgue measurable subsets of $\reals^1$
satisfying $|E_j|=e_j$ for each $j\in J$. Then
$\scriptt_\scriptl(\bE)=\scriptt_\scriptl(\bE^\star)$
if and only if there exists $\bv\in\reals^m$ satisfying
\begin{equation} E_j = E_j^\star+L_j(\bv) \end{equation}
for every $j \in J$.
\end{theorem}

Thus maximizing tuples $\bE$ are unique, up to the action of the symmetry
group $\reals^m$.
Here, and throughout the presentation, 
two sets are considered to be equal if their symmetric difference is a Lebesgue null set.
Thus the conclusion is that there exists $\bv$ such that for every $j\in J$,
\begin{equation} |E_j\symdif (E_j^\star+L_j(\bv))|=0.\end{equation}

The uniqueness statement can be put into more quantitative form 
in terms of the distance from $\bE$ to $\scripto(\bEstar)$.

\begin{theorem} \label{thm:main2}
Let $d,m,J,\scriptl,\scriptt_\scriptl$ be as in Theorem~\ref{thm:main}.
Let $S$ be a compact subset of $(0,\infty)^J$
such that every $\be\in S$ satisfies the hypotheses
of Theorem~\ref{thm:main}.
Then there exists $c>0$ such that for every $\be\in S$, and for every
$J$-tuple $\bE$ of Lebesgue measurable subsets of $\reals^1$
satisfying $|E_j|=e_j$ for each $j\in J$,
\begin{equation} \label{eq:mainconclusion}
\scriptt_\scriptl(\bE)\le \scriptt_\scriptl(\bE^\star) - c\dist(\bE,\scripto(\bE^\star))^2.
\end{equation} \end{theorem}

The exponent $2$ in the conclusion is optimal.
This inequality is not scale-invariant, but this is no contradiction
the hypothesis of compactness of $S$ precludes free scaling.

Theorem~\ref{thm:main} is an immediate consequence of Theorem~\ref{thm:main2}.
Indeed, it is elementary that if $E\subset\reals$
is a Lebesgue measurable set satisfying $0<|E|<\infty$,
and if  $\inf_I |E\symdif I|=0$
where the infimum is taken over all intervals $I\subset\reals$
satisfying $|I|=|E|$, then there exists an interval $I$ such that
$E=I$, that is, $|E\symdif I|=0$.
One proof is that since the mapping $t\mapsto |E\cap(I+t)|$ is continuous,
it assumes its minimum value.
Alternatively, $|I\symdif I'|\le |I\symdif E| + |E\symdif I'|$.
Therefore if $|I_n\symdif E|\to 0$, then the centers of the intervals
$I_n$ form a Cauchy sequence. 
We will prove Theorem~\ref{thm:main2} directly, and deduce Theorem~\ref{thm:main}
as a corollary.

It was shown in \cite{christgowersnorms} that for Gowers norms,
that is, for Gowers forms 
involving sets satisfying $E_\alpha=E_\beta$
for all $\alpha,\beta \subset\{0,1\}^k$,
or more generally for sets whose measures satisfy 
$|E_\alpha|=|E_\beta|$ for all $\alpha,\beta \subset\{0,1\}^k$,
the conclusion of Theorem~\ref{thm:main2} holds. In that situation, the genericity
hypothesis is violated; in fact, no extreme points are generic. 
(However, for all $\be$ outside a lower-dimensional set, 
the Gowers forms do satisfy the genericity hypothesis.)
Thus that hypothesis is superfluous in at least one situation. 

The case $m=2$ of Theorem~\ref{thm:main}
seems to be simpler than the general case. It
was treated in \cite{christflock} by an extension of the analysis
of Burchard \cite{burchard}, assuming $\be$ to be admissible
but not necessarily strictly admissible. We have not been able to
treat the case $m>2$ by that same method. For 
$m=2$, the genericity assumption  is a consequence of strict admissibility,
and the hypothesis on $K_j$ is also redundant.
Theorem~\ref{thm:main2} is new, even for $m=2$, except in special
cases such as $\Phi_{\text{RS}}$.


The genericity hypothesis is not natural in this theory, but simplifies
considerations. It is used principally in a step of the proof
of Proposition~\ref{prop:intervals}, which treats the case in which each set $E_j$
is an interval, but the centers of these intervals are arbitrary. 
It is also invoked in the proofs of Lemmas~\ref{lemma:LijLipschitz}
and \ref{lemma:quadraticexpansion}.
It is conceivable that a more careful execution of those proofs 
could remove this hypothesis.

The following nonquantitative uniqueness result for tuples $\bE$ of intervals is easy to establish,
under less restrictive hypotheses than those of Theorem~\ref{thm:main}.
It is not part of the development of our main theorems, but merits notice.

\begin{proposition} \label{prop:weakforintervals}
Let $d=1$.
Let $I_j\subset\reals$ be closed intervals of positive finite lengths centered at $0$.
For $\bv\in \reals^J$, define $\Psi(\bv)$ by \eqref{Psidefn}.
Set 
$\be = (|I_j|: j\in J)$. 
If $\scriptl$ is nondegenerate and if $(\scriptl,\be)$ is admissible, then
for any $\bv\in\reals^J$, $\Psi(\bv) = \Psi(\bzero)$ if and only if there exists $\by\in\reals^m$
such that $L_j(\by)=v_j$ for every $j\in J$.
\end{proposition}

\medskip
The method of proof of the two main results is as follows.
It suffices to establish Theorem~\ref{thm:main2}.
There exists a measure-preserving flow on $J$--tuples of sets,
under which the functional $\scriptt_\scriptl$ is nondecreasing and varies continuously. 
Therefore it suffices to establish \eqref{eq:mainconclusion}
for small perturbations of intervals centered at the origin. 
That is, it suffices to
prove that there exists $\delta_0>0$ such that
$\scriptt_\scriptl(\bE)\le\scriptt_\scriptl(\bE^\star)-c\dist(\bE,\scripto(\bE^\star))^2$
whenever $\be\in S$,
$\dist(\bE,\scripto(\bE^\star))\le\delta_0$,
and the other hypotheses of Theorem~\ref{thm:main2} are satisfied.

We expand the functional $\bE\mapsto \scriptt_\scriptl(\bE)$
in a perturbative series about $\bEstar$,
initially to first order and subsequently to second order,
and more generally about $(E_j^\star+L_j(\bv): j\in J)$,
with $\bv$ chosen to optimize the information obtained. We first 
use such an expansion to show that each $E_j$ has small symmetric
difference with an interval of length $e_j$.
It then remains to control the relative locations of the centers of these approximating intervals.

The case in which all sets $E_j$ are equal to intervals is analyzed separately, 
using convex geometry and ideas related to the Brunn-Minkowski inequality.
In a simple final step, these two complementary analyses are combined
to establish the full result.

The author is indebted to Kevin O'Neill for useful comments on the exposition.

\section{A flow of sets}

\begin{proposition}\label{prop:flow}
There exists a flow $(t,E)\mapsto E(t)$ of equivalence classes of Lebesgue measurable subsets
of $\reals^1$ with finite measures, defined for $t\in[0,1]$,
having the following properties for all equivalence classes of Lebesgue measurable subsets of $\reals$
with finite, positive measures.
\begin{enumerate}
\item
$E(0)=E$ and $E(1) = E^\star$.
\item 
Preservation of measure:
$|E(t)| = |E|$ for all $t\in[0,1]$.
\item
Continuity:
$|E(s)\symdif E(t)|\to 0$ as $s\to t$.
\item
Inclusion monotonicity:
If $E\subset\tilde E$ then $E(t)\subset \tilde E(t)$ for all $t\in[0,1]$.
\item
Contractivity: 
$|E_1(t)\symdif E_2(t)|\le |E_1\symdif E_2|$ for all sets $E_1,E_2$ and all $t$.
\item
Independence of past history:
If $0\le s\le t\le 1$ then
$E(t)$ depends only on $E(s),s,t$.
\item Functional continuity:
$\scriptt_\scriptl(\bE(s))\to\scriptt_\scriptl(\bE(t))$ as $s\to t$.
\item 
Functional monotonicity:
If $E_j\subset\reals$ are measurable sets with $|E_j|\in(0,\infty)$ then
the function $t\mapsto \scriptt_\scriptl(\bE(t))$ is nondecreasing on $[0,1]$.
\end{enumerate}
\end{proposition}

Here $\bE(t)$ denotes $(E_j(t): j\in J)$.

All of these statements are to be interpreted in terms of equivalence classes
of measurable sets, with $E$ equivalent to $E'$ whenever $|e\symdif e'|=0$.
Thus $E(t)\subset\tilde E(t)$ means that $\tilde E(t)\setminus E(t)$ is a Lebesgue null set.
In the case in which the initial set $E$ is a finite union of pairwise disjoint closed intervals,
this flow is a well known device \cite{BLL}, \cite{liebloss}. 
Except for the functional continuity and monotonicity conclusions, Proposition~\ref{prop:flow}
is proved in \cite{christRSult}. Functional monotonicity follows from 
contractivity and inclusion monotonicity, together with the functional
monotonicity for finite unions of intervals established by Brascamp, Lieb, and Luttinger
\cite{BLL}, in exactly the same way that the corresponding functional monotonicity was
established in \cite{christRSult}.
Functional continuity is a consequence of the next lemma.
\qed

\begin{lemma} \label{lemma:continuity}
If $\scriptl$ is nondegenerate then
there exist exponents $\gamma_j\in(0,1]$ and $C<\infty$ such that for every $J$-tuple $\bE$
of Lebesgue measurable sets,
\begin{equation} \scriptt_\scriptl(\bE) \le C\prod_{j\in J} |E_j|^{\gamma_j}. \end{equation}
\end{lemma}

\begin{proof}
If $J'\subset J$ has cardinality $m$
and $\{L_i: i\in J'\}$ is a basis for $(\reals^m)^*$
then there exists $C_{J'}<\infty$ such that 
\begin{equation} \label{eq:trivialmultibound}
\scriptt_\scriptl(\bE) \le C_{J'} \prod_{i\in J'} |E_i|
\ \text{ for all $\bE$.} \end{equation}

According to the nondegeneracy hypothesis,
$\{L_j: j\in J\}$ spans the dual space $(\reals^m)^*$,
and none of these vanish. Therefore for each $j\in J$ there exists
such a subset $J'\subset J$ that contains $j$ and forms a basis for $(\reals^m)^*$.
Thus there exists a finite collection of subsets $J'\subset J$,
satisfying \eqref{eq:trivialmultibound},
such that each $j\in J$ belongs to at least one of these. 
In the geometric mean of the right-hand sides of all associated inequalities
\eqref{eq:trivialmultibound}, $|E_j|$ is raised to a positive power for each index $j$. 
Thus we arrive at the conclusion of the lemma.
\end{proof}

We record a related fact that will be used below.

\begin{lemma} \label{lemma:addon}
Let $i\ne j\in\scriptl$.  There exists $C<\infty$ such that
for all functions $f_n\in L^1\cap L^\infty$,
\begin{equation}
\big|\Phi_\scriptl(\bff)\big| \le C\norm{f_i}_1\norm{f_j}_1\prod_{k\ne i,j} \min(\norm{f_k}_1,\norm{f_k}_\infty).
\end{equation}
\end{lemma}

\begin{proof}
By hypothesis, $L_i$ and $L_j$ are not colinear, hence are linearly independent.
Hence there exists a linearly independent subset $J'\subset J$ of cardinality $m$
that contains both $i$ and $j$.
Then
\[ \Phi(\bff) \le C \prod_{k\in J'} \norm{f_k}_1 \cdot \prod_{n\in J\setminus J'} \norm{f_n}_\infty.\]
\end{proof}

Proposition~\ref{prop:flow} is not genuinely needed in our proofs;
it suffices to prove Theorem~\ref{thm:main2} for sets that are finite unions of intervals.
The flow for those sets was constructed in \cite{BLL}, and is all that our method requires
to analyze them.
That such a flow could be extended to general measurable sets seems to have been known
\cite{burchardoral},\cite{burchardlecturenotes}, 
though perhaps not widely documented in the literature.

\section{Analysis for intervals}

In this section we analyze the situation in which the sets $E_j$ are all intervals.
Let $m\ge 2$, let $J$ be a finite index set of cardinality $|J|>m$,
and for each $j\in J$ let $L_j:\reals^m\to\reals^{1}$ be a surjective linear mapping.
Let $I_j$ be closed intervals in $\reals$ 
centered at $0$, 
of finite, positive lengths $|I_j|=e_j$.
For $\bv\in\reals^J$ define
\begin{equation} \label{Psidefn}
\Psi(\bv) = \Phi_\scriptl(I_j+v_j: j\in J) =  \int_{\reals^m} \prod_{j\in J} \one_{I_j+v_j}(L_j(\bx))\,d\bx.
\end{equation}
By the Brascamp-Lieb-Luttinger inequality \eqref{eq:BLL}, $\Psi(\bv)\le\Psi(\bzero)$ for all $\bv$.
Alternatively, this is a consequence of the Brunn-Minkowski inequality; see below.
A sufficient condition for equality is that there exist
$\by\in\reals^m$ satisfying $L_j(\by)=v_j$ for all $j\in J$,
for the substitution $\bx\mapsto \bx-\by$ reduces $\Psi(\bv)$ to $\Psi(\bzero)$.
These are $|J|$ linear equations in $m<|J|$ variables. 



To prepare for the proof of Proposition~\ref{prop:weakforintervals},
define the convex set $K\subset\reals^m\times\reals^J$ by
\begin{equation} K = \{(\bx,\bu)\in\reals^m\times\reals^J: 
L_j(\bx)\in I_j+u_j \ \text{ for all $j\in J$.} \end{equation}
Equivalently, $|L_j(\bx)-u_j|\le \tfrac12|I_j|$.
For $\bu\in \reals^J$ define
\begin{equation} K(\bu)=\{ \bx\in\reals^m: (\bx,\bu)\in K\}. \end{equation}
$\Psi(\bu)$ represents the $m$--dimensional volume $|K(\bu)|$ of $K(\bu)$.
$|K(-\bu)|\equiv|K(\bu)|$, by the change of variables $\bx\mapsto -\bx$ in $\reals^m$,
since $I_j$ is centered at the origin.
Since $K$ is convex,
\begin{equation} \label{containment}
K(\bzero)\supset \tfrac12 K(-\bv) + \tfrac12K(\bv) \qquad\forall\,\bv\in \reals^J.  \end{equation}
Therefore by the Brunn-Minkowski inequality,
\begin{equation} \label{eq:KBM} |K(\bzero)| \ge |K(\bv)|^{1/2} |K(-\bv)|^{1/2} = |K(\bv)|.\end{equation}

\begin{proof}[Proof of Proposition~\ref{prop:weakforintervals}]
Suppose that $\bv\in\reals^J$ satisfies
$\Psi(\bv)=\Psi(\bzero)$.
Then since $|K(-\bv)| = |K(\bv)|$, $|K(\bzero)| = |K(\bv)| = |K(\bv)|^{1/2}|K(-\bv)|^{1/2}$,
and hence by \eqref{eq:KBM} and the arithmetic-geometric mean inequality,
\begin{equation} |K(\bzero)| \ge \big|\tfrac12 K(-\bv)+\tfrac12 K(\bv)\big| \ge 
|K(\bv)|^{1/2}|K(-\bv)|^{1/2} = |K(\bzero)|. \end{equation}
Thus
\begin{equation} \label{equalityholds} 
\big|\tfrac12 K(-\bv)+\tfrac12 K(\bv)\big| = 
|K(\bv)|^{1/2}|K(-\bv)|^{1/2}
= |K(\bv)| = |K(\bzero)|. \end{equation}
According to the well-known characterization of cases of equality in the Brunn-Minkowski inequality,
the three sets $K(\bv)$, $K(-\bv)$, and $K(\bzero)$ must be translates of one another.
So there exists $\by\in\reals^m$ such that $K(\bv) + \by = K(\bzero)$. 

We claim that $v_j= L_j( - \by)$ for every index $j\in J$.
Indeed, the relation $K(\bv)=K(\bzero)-\by$ means that for every $j$,
for any $\bx\in\reals^m$,
\[ L_j(\bx+\by)\in I_j \Leftrightarrow L_j(\bx)\in I_j+v_j. \]
Thus by substituting $\bx=\bz-\by$ we find that for any $\bz\in\reals^m$,
\begin{equation}\label{implication} \big[|L_j(\bz)| \le \tfrac12|I_j|\ \forall\,j\in J\big]
\Longrightarrow
\big[|L_j(\bz) -v_j-L_j(\by)|\le \tfrac12|I_j|\ \forall\, j\in J \big]. \end{equation}

Let $k\in J$.
By the admissibility hypothesis,
there exists $\bx\in\reals^m$ such that $|L_j(\bx)|\le \tfrac12 |I_j|$
for every $j\in J$, and $L_k(\bx)  = |I_k|/2$.  According to \eqref{implication}
applied both with $\bz=\bx$ and with $\bz = -\bx$,
$|L_k(\bx)-v_k-L_k(\by)|\le|I_k|/2$ 
and
$|L_k(-\bx)-v_k-L_k(\by)|\le|I_k|/2$. 
If $v_k+L_k(\by)<0$ then 
\[ L_k(\bx)-v_k-L_k(\by) = \tfrac12|I_k|-(v_k+L_k(\by)) > \tfrac12|I_k|,\]
contradicting \eqref{implication} for $\bx$.
In the same way, if $v_k+L_k(\by)>0$ then a contradiction is reached for $-\bx$.
Therefore $v_k=-L_k(\by) = L_k(-\by)$. 
\end{proof}

The next result is the main goal of this section.
Only in its proof is the genericity hypothesis invoked.

\begin{proposition} \label{prop:intervals} 
Let $d = 1$.
Let $m,J,\scriptl,\scriptt_\scriptl,S$ satisfy the hypotheses of Theorem~\ref{thm:main2}.
There exists $c>0$
such that for every 
$\be\in S$ and every
$J$-tuple $\bI$ of intervals $I_j\subset\reals$
satisfying $|I_j|=e_j$ for each $j\in J$,
\begin{equation}
\scriptt_\scriptl(\bI)\le \scriptt_\scriptl(\bI^\star) - c\dist(\bI,\scripto(\bI^\star))^2.
\end{equation}
\end{proposition}

To each extreme point $\bp$ of $\scriptk_\be = K(\bzero) \subset\reals^m$ we associate
$J_\bp$, the set of all $j\in J$ such that $|L_j(\bp)| = e_j/2$.
$\{L_j: j\in J_\bp\}$ must span $(\reals^m)^*$;
otherwise $\bp$ could not be an extreme point.
The genericity hypothesis states that every $J_\bp$ has cardinality equal to $m$,
so $J_\bp$ must be linearly independent.

$\scriptk_\be$ is a compact convex polytope. 
Define $\fg$ to be the (undirected) graph whose vertices are
the extreme points of $\scriptk_\be$, 
and whose edges are the line segments of the $1$-skeleton of this polytope.
If an extreme point $\bp$ is generic in the sense defined above,
then $m$ segments of the $1$-skeleton contain $\bp$, and these 
are contained in the translates by $\bp$
of the lines defined by intersections of nullspaces of $m-1$ elements of $J_\bp$. 
Two distinct extreme points $\bp,\bq$ are adjacent in this graph if and only if either 
$J_\bp=J_\bq$ and $L_j(\bp)=L_j(\bq)$ for exactly $m-1$ indices $j\in J_\bp$,
or
$J_\bp\cap J_{\bq}$ has cardinality equal to $m-1$
and $L_j(\bp)=L_j(\bq)$ for every $j\in J_\bp\cap J_\bq$.
In the latter situation,
$\bigcap_{j\in J_\bp\cap J_\bq} \kernel(L_j)$ has dimension equal to $1$, 
and $|L_k(\bp)|\ne |L_k(\bq)|$ if 
$k\in L_\bp\symdif L_\bq$.


If 
$\{ L_j: j\in J'\}\subset\{L_j: j\in J\}$
spans $(\reals^m)^*$,
and if $\bx,\by\in\reals^m$ satisfy $L_j(\bx)=L_j(\by)$
for every $j\in J'$, then $\bx=\by$. Thus for distinct extreme points $\bp,\bq$, 
it is not possible to have $L_j(\bp)=L_j(\bq)$ for every $j\in L_\bp$.

\begin{lemma}
The graph $\fg$ is connected.
\end{lemma}

\begin{proof}
Given any
two extreme points, the line segment joining them lies in the convex set $\scriptk_\be$. 
Viewed as a piecewise affine path, this segment can be continuously deformed within $\scriptk_\be$ to lie
in progressively lower-dimensional faces until it lies in the $1$-skeleton.
\end{proof} 

If $\bw\in\reals^J$ and $|\bw|$ is sufficiently small then as a consequence of the genericity
hypothesis, the extreme points of $K(\bw)$ are in natural one-to-one correspondence
with the extreme points of $\scriptk_\be = K(\bzero)$, and each extreme
point of $K(\bw)$ remains close to a unique extreme point of $K(\bzero)$.
Each extreme point of $K(\bw)$ can thus be regarded as a continuous function $\bp(\bw)$ of $\bw$.

\begin{lemma} \label{lemma:adjacent}
Uniformly for all sufficiently small $\bw\in\reals^J$, 
for every pair of extreme points $\bp,\bq$ of $\scriptk_\be$ that are adjacent in $\fg$, 
\begin{equation}\label{eq:Kbwgain1} 
|K(\bw)| \le |K(\bzero)|-c|(\bp(\bw)-\bp(\bzero))-(\bq(\bw)-\bq(\bzero))|^2.  \end{equation}
\end{lemma}

If there exists $\bv\in\reals^m$ satisfying $L_j(\bv)=w_j$ for every $j\in J$
then $\bp(\bw)-\bp(\bzero)=\bv$ for every extreme point $\bp$ of $\scriptk_\be$,
so \eqref{eq:Kbwgain1} asserts mererly that $|K(\bw)|\le |K(\bzero)|$.
Indeed, $|K(\bw)| = |K(\bzero)|$ in that case.

\begin{proof}[Proof of Lemma~\ref{lemma:adjacent}]
Let $\bp\ne\bq$ be adjacent vertices of $\fg$.
Consider first the case in which $J_\bp=J_\bq$.
Let $\bw\in\reals^J$.
Then $L_j(\bp(\bw))= L_j(\bp(\bzero))+w_j$ for all $j\in J_\bp$, 
and likewise $L_j(\bq(\bw))= L_j(\bq(\bzero))+w_j$ for all $j\in J_\bq = J_\bp$. 
Consequently $L_j(\bp(\bw)-\bq(\bw))\equiv L_j(\bp(\bzero)-\bq(\bzero))$
for all $j\in J_\bp$. Therefore
since $\{L_j: j\in J_\bp\}$ spans the dual space $\reals^{m*}$,
$\bp(\bw)-\bq(\bw)=\bp(\bzero)-\bq(\bzero)$
The conclusion of the lemma then holds trivially, since
$|K(\bw)|\le |K(\bzero)|$
by the Brunn-Minkowski inequality. 

Consider next any pair $\bp,\bq$ of adjacent vertices for which
$J_\bp\cap J_\bq$ has cardinality $m-1$.
By translating in $\reals^m$ we may assume without loss of generality that $w_j=0$
for every $j\in J_\bp$. Then $\bp(\bw)=\bp(\bzero) = \bp$,
while $\bq(\bw)-\bq(\bzero) = (\bq(\bw)-\bp(\bw)) - (\bq(\bzero)-\bp(\bzero))$ 
is an element of the nullspace of $L_i$
for every $i\in J_\bp\cap J_\bq$. Thus $\bq(\bw)-\bq(\bzero))$
is an element of a  one-dimensional subspace that is independent of $\bw$.

By renaming indices, making a linear change of variables in $\reals^m$,
and replacing $L_j$ by $\pm 2e_j^{-1}L_j$ for all $\bx\in\reals^m$ for each $j\in  J_\bp$,
we may reduce matters to the situation in which
$J_\bp = \{1,2,\dots,m\}$, 
$J_\bq\cap J_\bp = \{1,2,\dots,m-1\}$,
$L_j(x) = x_j$ for every $x\in\reals^m$ for each $j\in J_\bp$,
$e_j=2$ for all $j\in J_\bp$,
$\bp = (-1,-1,\dots,-1)$,
and $\bq = (-1,-1,\dots,-1,a)$ for some $a>-1$.
Then there exists a neighborhood of the line segment joining $\bp$ to $\bq$
in which $|L_j|$ is strictly less than $e_j/2$
for every index $j\notin J_\bp\cup J_\bq$.
Indeed, suppose that $|L_j(z)|=e_j/2$ for some point $z$ of this segment.
$z$ cannot equal $\bp$ or $\bq$, since $j\notin J_\bp\cup J_\bq$.
$L_j$ cannot be constant in a neighborhood of $z$ on the segment, for then
it would be constant on the whole segment and hence $|L_j(\bp)|=|L_j(z)| = e_j/2$,
a contradiction. But if $L_j$ is not constant on the segment then
since $z$ is an interior point, $|L_j|$ attains a strictly larger
value at some other point of the segment,
contradicting the fact that $|L_j|\le e_j/2$ at every point of $K(\bzero)$.

With these choices and reductions, $\bp(\bw)\equiv\bp(\bzero)=\bp$, while the point
$\bq(\bw)$ takes the form $(-1,-1,\dots,-1,a(\bw))$,
for all sufficiently small $\bw\in\reals^J$.

Define $\bz$ to be the vector $\bz = (1,1,\dots,1,0)$.
Then $\bp+t\bz$ belongs to the boundary of $K(\bzero)$ 
for every sufficiently small $t>0$,
and $\bp+t\bz + (0,0,\dots,0,s)$ belongs to the interior of $K(\bzero)$
for all sufficiently small $t,s>0$.
For small positive $s\in\reals$ consider the halfspaces 
\begin{equation} \label{halfling}
H_s=\{\bx\in\reals^m: \langle \bx,\bz\rangle \le \langle \bp,\bz\rangle+s\}.\end{equation}
$H_s\cap K(\bzero)$ has positive Lebesgue measure for every $s>0$.
For small $s>0$, $H_s\cap K(\bzero)$ contains
a small neighborhood in $K(\bzero)$ of the line segment whose endpoints are $\bp,\bq$.
Conversely, any such neighborhood contains $H_s\cap K(\bzero)$ for all sufficiently
small $s>0$.
Choose and fix $s>0$ sufficiently small to ensure that
$|L_j|< e_j/2$ in a neighborhood of $H_s\cap K(\bzero)$
for every $j\notin J_\bp\cup J_\bq$.

Set
\[\theta = \frac{|H_s\cap K(\bzero)|}{|K(\bzero)|} \in (0,1).\]
For each sufficiently small vector $\bw\in\reals^m$ there exist unique $t=t(\bw),t'=t(-\bw)\in\reals^+$ 
satisfying $|H_t\cap K(\bw)| = \theta|K(\bw)|$
and likewise $|H_{t'}\cap K(-\bw)| = \theta |K(-\bw)|$. 
These parameters vary continuously with $\bw$, and satisfy $t=t'=s$ when $\bw=0$. 

Now $\tf K(\bw)+ \tf K(-\bw)$ contains the union of the two convex sets
$\tf(K(\bw)\cap H_t)+\tf(K(-\bw)\cap H_{t'})$
and
$\tf(K(\bw)\setminus H_t)+\tf(K(-\bw)\setminus H_{t'})$.
These two sets are disjoint except for their boundaries, so the measure
of their union equals the sum of their measures.

By the Brunn-Minkowski inequality,
\begin{align*} \big|\tf(K(\bw)\setminus H_t)+\tf(K(-\bw)\setminus H_{t'}) \big|
& \ge |K(\bw)\setminus H_t|^{1/2}|K(-\bw)\setminus H_{t'}|^{1/2}
\\ &
= (1-\theta)|K(\bw)|
\end{align*}
since $|K(-\bw)|=|K(\bw)|$.
If we can show that
\begin{multline} \label{eq:ifwecan}
|\tf(K(\bw)\cap H_t)+\tf(K(-\bw)\cap H_{t'})|
\\
\ge |K(\bw)\cap H_t|^{1/2}|K(-\bw)\cap H_{t'}|^{1/2}
+ c|\bq(\bw)-\bq(-\bw)|^2
\end{multline}
then since  the right-hand side is equal to $\theta|K(\bw)| + c|\bq(\bw)-\bq(-\bw)|^2$
by our choices of $t,t'$, we may conclude that
\[ |\tf K(\bw)+\tf K(-\bw)| \ge |K(\bw)| + c|\bq(\bw)-\bq(-\bw)|^2\]
and  consequently, since $\tf K(\bw)+\tf K(-\bw) \subset K(\bzero)$,
\begin{equation} \label{eq:Kbwmw} |K(\bw)| \le |K(\bzero)|-c|\bq(\bw)-\bq(-\bw)|^2.\end{equation}
Now 
\begin{equation} [\bq(\bw)-\bq(-\bw)] = 2[\bq(\bw)-\bq(\bzero)],\end{equation}
so \eqref{eq:Kbwmw} is equivalent to
\begin{equation}\label{eq:Kbwmz} |K(\bw)| \le |K(\bzero)|-c|\bq(\bw)-\bq(\bzero)|^2\end{equation}
with a different value of $c>0$.
We have normalized so that $\bp(\bw)-\bp(\bzero)=0$, so 
\eqref{eq:Kbwmz} is a restatement of the conclusion of Lemma~\ref{lemma:adjacent}.
\end{proof}

\begin{proof}[Proof of \eqref{eq:ifwecan}]
Let $J_\bq\setminus J_\bp = \{k\}$.
If $t>0$ is sufficiently small then
$K(\bw)\cap H_t$ is the set of all $\bx\in\reals^m$
that are close to the line segment with endpoints $\bp,\bq$ and satisfy
$x_j\ge -1$ for all $j\le m$,
$\sum_{j=1}^{m-1} x_j \le -(m-1) +t$,
and $|L_k(\bx)-w_k|\le e_k/2$.
Since $k\notin J_\bp$ and $\bp\in K(\bzero)$, $|L_k(\bp)|<e_k/2$. 
Therefore, after possibly replacing $L_k$ by $-L_k$, 
$K(\bw)\cap H_t$ is equal to the set of all $\bx\in\reals^m$ that satisfy
$\sum_{j=1}^{m-1} x_j < -(m-1)+t$,
$x_j\ge -1$ for all $j\le m$,  and $ L_k(\bx) \le \tf e_k + w_k$.
The function $\bx\mapsto L_k(\bx)$ cannot be independent of the final coordinate $x_m$, since 
$L_\bq = \{L_1,L_2,\dots,L_{m-1}\}\cup\{L_k\}$ spans $(\reals^m)^*$
and $L_j(\bx)\equiv x_j$ for $1\le j\le m-1$. 

Without loss of generality, we may multiply $L_k,e_k$ by constants
to put $L_k$ into the form $L_k(\bx) = x_m + \ell(\bx')$,
where $\ell:\reals^{m-1}\to\reals$ is linear and 
$\bx = (\bx',x_m)\in\reals^{m-1}\times\reals$.
Of the inequalities $|L_j(\bx)|\le e_j/2$ with $j\in J_\bp\cup J_\bq$,
only those with $j=m$ and $j=k$ involve the coordinate $x_m$.
These inequalities together take the form $-1\le x_m \le \tf e_k + w_k - \ell(\bx')$
in a neighborhood of the segment joining $\bq$ to $\bp$.
Thus for every point $\bx'$ sufficiently close to $(-1,-1,\dots,-1)\in\reals^{m-1}$,
$\{u\in\reals: (\bx',u)\in K(\bw)\cap H_t\}$
is a line segment of length 
\begin{equation}\label{eq:fwxprime} f_\bw(\bx') = \tf e_k +1-\ell(\bx') + w_k.\end{equation} 
The difference between this length and the length of the corresponding
line segment for $K(-\bw)$ is equal to $\pm 2w_k$, a quantity independent of $\bx'$.
In particular, $|\bq(\bw)-\bq(-\bw)| = 2|w_k|$.
Likewise, $|\bq(\bw)-\bq(\bzero)| = |w_k|$.
Moreover,
the difference between the ratio of these two lengths, and $1$, 
also has magnitude comparable to $|w_k|$.
Thus the conclusion \eqref{eq:Kbwgain1} of Lemma~\ref{lemma:adjacent} is equivalent to
\begin{equation}\label{eq:Kbwgain2} 
|K(\bw)| \le |K(\bzero)|-c|(\bp(\bw)-\bp(-\bw))-(\bq(\bw)-\bq(-\bw))|^2.  \end{equation}

Continuing to regard $\reals^m$ as $\reals^{m-1}\times\reals$,
define $K'(\bw)$ to be the projection onto $\reals^{m-1}$ of $K(\bw)\cap H_{t(\bw)}$.
Thus for all sufficiently small vectors $\bw$,
$K'(\bw)$ is the set of all $\bx'=(x_1,\dots,x_{m-1})\in\reals^{m-1}$ 
satisfying $x_j+1\ge 0$ for all $j\in\{1,2,\dots,m-1\}$ and $\sum_{j=1}^{m-1}(x_j+1)\le t(\bw)$.
For all $\bw$ sufficiently close to $\bzero$, $K'(-\bw)$ is homothetic to $K'(\bw)$. 
Let $\bp'=(-1,-1,\cdots,-1)$
be the projection of $\bp$ onto $\reals^{m-1}$.  Define
the homothety $\phi:K'(\bw)\to K'(-\bw)$ by $\phi(\bp'+\bx') = \bp'+s\bx'$
where $s=s(\bw) = t(-\bw)/t(\bw) \in\reals^+$ is chosen so that $\phi$ is a bijection.
Then
\begin{equation} C^{-1}|w_k|\le |s-1|\le C|w_k| \end{equation}
for some constant $C\in\reals^+$, uniformly in $\bw$ provided that $|\bw|$
is sufficiently small, by the observation concerning the ratio of
lengths made above.

Define $\tilde K(\bw) = K(\bw)\cap H_{t(\bw)}$ 
and $\tilde K(-\bw) = K(-\bw)\cap H_{t(-\bw)}$. 
We claim that there exists $c>0$, depending only on $m,\scriptl,\be$, such that
\begin{equation} \label{eq:convexgeometry}
\big|\tf \tilde K(\bw) + \tf \tilde K(-\bw)\big|
\ge (1+c|w_k|^2)  |\tilde K(\bw)|.
\end{equation}

To prove this claim, consider the one-dimensional Lebesgue measure
$f_\bw(\bx') = \tfrac12 e_k +1 - \ell(\bx') + w_k$ of the set of all $y\in\reals$
such that $(\bx',y)\in \tilde K(\bw)$.
The set $\tf \tilde K(\bw) + \tf \tilde K(-\bw)$
contains all points $\tf (\bx',u) + \tf (\phi(\bx'),v)$
such that $(\bx',u)\in \tilde K(\bw)$ and $(\phi(\bx'),v)\in \tilde K(-\bw)$.
Thus
$\tf \tilde K(\bw) + \tf \tilde K(-\bw)$
contains the set of all points $(\tf \bx' + \tf \phi(\bx'),\tf u + \tf v)$
where $\bx',u,v$ are as above.
Therefore 
the set of all $y\in\reals$ such that
$(\tf\bx' + \tf\phi(\bx'),y)$ belongs to 
$\tf \tilde K(\bw) + \tf \tilde K(-\bw)$
has one-dimensional Lebesgue measure greater than or equal to
$\tf f_\bw(\bx') + \tf f_{-\bw}(\phi(\bx'))$.

Therefore, since the Jacobian determinant of the map
$\bx'\mapsto \tfrac12 \bx' + \tfrac12 \phi(\bx')$ is equal
to $2^{-(m-1)}(1+s)^{m-1}$,
\begin{align*}
\big|\tf \tilde K(\bw) + \tf \tilde K(-\bw)\big|
\ge \int_{K'(\bw)} \big(
\tf f_\bw(\bx') + \tf f_{-\bw}(\phi(\bx'))
\big) 2^{-(m-1)}(1+s)^{m-1}\,d\bx'.
\end{align*}
Split this into two terms. The first of these is
\[ 2^{-m} (1+s)^{m-1} \int_{K'(\bw)}  f_\bw(\bx') \,d\bx'
=  2^{-m} (1+s)^{m-1} |\tilde K(\bw)| .\]
The second is
\begin{align*}
 2^{-m} (1+s)^{m-1} \int_{K'(\bw)}  f_{-\bw}(\phi(\bx')) \,d\bx'
&=  2^{-m} (1+s)^{m-1} \int_{K'(-\bw)}  f_{-\bw}(\bx') s^{-(m-1)} \,d\bx'
\\& =  2^{-m} (1+s)^{m-1} s^{-(m-1)} |\tilde K(-\bw)|  
\\& =  2^{-m} (1+s)^{m-1} s^{-(m-1)} |\tilde K(\bw)|  
\end{align*}
since $|\tilde K(-\bw)| = |\tilde K(\bw)|$.
Recombining these two results 
and using the inequality
\[ 2^{-m} (1+s)^{m-1} \big( 1+ s^{-(m-1)} \big) \ge (1+c(s-1)^2),\] 
where $c>0$ depends only on the dimension $m$, gives
\begin{equation}
\big|\tf \tilde K(\bw) + \tf \tilde K(-\bw)\big|
\ge (1+c(s-1)^2)  |\tilde K(\bw)|
\end{equation}
provided that $|\bw|$ is small.
Since $|s-1|$ is comparable to $|w_k|$, we have established the claim \eqref{eq:convexgeometry}.
Since $|w_k|$ is in turn comparable to $|\bq(\bw)-\bq(-\bw)|$,
\eqref{eq:ifwecan} follows from \eqref{eq:convexgeometry}. 
This completes the proof of Lemma~\ref{lemma:adjacent}.
\end{proof}

\begin{proof}[Proof of Proposition~\ref{prop:intervals}]
Denote by $S$ the set of all ordered pairs of adjacent vertices $(\bp,\bq)$ in $\fg$. 
Define a mapping $T$ from a neighborhood of $\bzero\in \reals^J$ to $(\reals^m)^S$ by
\begin{equation} T(\bw)(\bp,\bq) = [\bq(\bw)-\bq(\bzero)]-[\bp(\bw) - \bp(\bzero)]\,\in\,\reals^m \end{equation}
for $(\bp,\bq)\in S$.
We have seen in the above discussion that $T$ depends linearly on $\bw$ in
a small neighborhood of $\bzero$. Denote also by the symbol $T$
its unique extension to a linear mapping from $\reals^J$ to $(\reals^m)^S$.
It suffices to show that the nullspace
of this extension $T$ is equal to the image of $\reals^m$ in $\reals^J$ under the mapping
$\bv\mapsto \bL(\bv) = (L_j(\bv): j\in J)$. We have already remarked, immediately after
the statement of Lemma~\ref{lemma:adjacent}, that this nullspace does contain $\bL(\reals^m)$.

Let $\bw$ be an element of this nullspace.
Fix any vertex $\bp_0$ of $\fg$.  
The set of vertices $\bp\in\fg$ satisfying
$\bp(\bw)-\bp(\bzero)=\bp_0(\bw)-\bp_0(\bzero)$ 
is connected, since it is given that
$\bq(\bw)-\bq(\bzero)=\bp(\bw)-\bp(\bzero)$ whenever $\bp,\bq$ are adjacent.
Since this set contains $\bp_0$, and since $\fg$ is connected,
it follows that
$\bp(\bw)-\bp(\bzero)=\bp_0(\bw)-\bp_0(\bzero)$ 
for every vertex $\bp$ of $\fg$.

Since $J_{\bp_0}$ is a basis for $(\reals^m)^*$,
there exists a unique $\bv\in\reals^m$
satisfying $L_j(\bv) = w_j$ for each $j\in J_{\bp_0}$.
Define $\bz=(z_j: j\in J)$ by $z_j = w_j-L_j(\bv)$.
Then $\bp(\bz)=\bp(\bzero)$ for every vertex $\bp$.
It suffices to show that $\bz=\bzero$, and of course,
by linearity, it suffices to show this under the assumption that $\bz$ is small.

As was shown above, if $\bp$ is an extreme point of $\scriptk_\be$
then $\bp(\bz)=\bp(\bzero)$ if and only if $z_j=0$ for every $j\in J_\bp$. 
The admissibility hypothesis guarantees that each index $j\in J$
belongs to $J_\bp$ for some extreme point.
Indeed, the intersection of $\scriptk_\be$ with
$\{\bx: L_j(\bx)=e_j/2\}$ is nonempty by the admissibility hypothesis.
This intersection is compact and convex,
so contains at least one extreme point,
and its extreme points are also extreme points of $\scriptk_\be$. 
Thus $z_j=0$ for every index $j\in J$.
Equivalently, $w_j=L_j(\bv)$ for every $j\in J$.
\end{proof}

\section{Perturbative expansion} \label{section:pert}

We adapt the approach developed in \cite{christRSult} (see also \cite{christinterval})
to analyze $\Phi_\scriptl(\bE)$, under the assumption that $\dist(\bE,\scripto(\bEstar))$ is small
relative to $\max_j |E_j|$.
Throughout the discussion, $\scriptl$ is considered to be fixed,
and $\Phi=\Phi_\scriptl$.

The goal of \S\ref{section:pert} is to prove Proposition~\ref{prop:nearlyintervals},
which asserts that if $\bE$ nearly maximizes $\Phi_\scriptl$,
then each set $E_j$ must nearly coincide with an interval.
Conclusions concerning the relative arrangement of the centers of these intervals 
will not be drawn until \S\ref{section:hybrid}.

\begin{proposition} \label{prop:nearlyintervals}
Let $m,J,\scriptl,\scriptt_\scriptl,S$ satisfy the hypotheses of Theorem~\ref{thm:main2}.
There exist $\delta_0>0$ and $C<\infty$ such that the following holds for every $\be\in S$.
Let $\bE$ be a $J$--tuple of Lebesgue measurable subsets of $\reals$
satisfying $|E_j|=e_j$ for each $j\in J$,
and satisfying $\dist(\bE,\scripto(\bEstar))\le\delta_0$.
Then for each $k\in J$ there exists an interval $I_k\subset\reals$ such that 
\begin{equation} \label{eq:nearlyintervalsconclusion} 
|E_j\symdif I_j|^2 \le C\,(\Phi(\bEstar)-\Phi(\bE))
+ C\dist(\bE,\scripto(\bEstar))^3.  \end{equation}
\end{proposition}

In \S\ref{section:hybrid} we will show that $\dist(\bE,\scripto(\bEstar))^2$
satisfies the same upper bound, allowing absorption of the cubic term
on the right-hand side of \eqref{eq:nearlyintervalsconclusion}  
into the left-hand side and thus completing the proofs of Theorems~\ref{thm:main}
and \ref{thm:main2}.

\subsection{Perturbation analysis: first order expansion}

Let $\bE$ be given. To simplify notation, define 
\begin{equation} \label{eq:dorbit} \delta = \dist(\bE,\scripto(\bEstar)).\end{equation}
Replace each $E_j$ by $E_j+L_j(\bv)$, where $\bv\in\reals^m$
is chosen so that $\max_j |(E_j+L_j(\bv))\symdif \Estar_j|\le 2\delta$.
Change notation, denoting $E_j+L_j(\bv)$ by $E_j$ and denoting $(E_j: j\in J)$
by $\bE$. Thus $|E_j\symdif \Estar_j|\le 2\dist(\bE,\scripto(\bEstar))$ for every $j\in J$.

Define $f_j:\reals\to\reals$ by 
\begin{equation}\label{eq:expandEj} \one_{E_j} = \one_{\Estar_j} + f_j.\end{equation}
$f_j$ vanishes on the complement of $E_j\symdif \Estar_j$,
satisfies $|f_j|\equiv 1$ on $E_j\symdif\Estar_j$,
and thus satisfies $\norm{f_j}_{L^1}=|E_j\symdif\Estar_j|\le 2\delta$.
It also satisfies $\int_{\reals^d} f_j=0$.
Inserting \eqref{eq:expandEj} in place of $\one_{E_j}$ for each index in the definition of
$\scriptt_\scriptl(\bE)$, then invoking the multilinearity of $\scriptt_\scriptl$,
yields an expansion of $\scriptt_\scriptl$ as a sum of $2^{|J|}$ terms.

The first-order terms in this expansion --- those that involve $f_j$ for a single 
index $j$ --- are $\langle K_j,f_j\rangle = \int_\reals K_j f_j$ where
$K_j$ are the kernels introduced in Definition~\ref{defn:K_j}.
Because the one-sided derivatives of $K_j$ are strictly negative at $e_j/2$
according to the strict admissibility hypothesis,
and because $K_j$ is nonincreasing on $[0,\infty)$
as shown in the discussion following Definition~\ref{defn:K_j},
\begin{equation} \label{eq:onesided} \langle K_j,f_j\rangle
\le -c\int \min(1,\big|\, |x|-e_j/2\,\big|) \cdot |f_j(x)|\,dx \end{equation}
for a certain constant $c>0$.
This holds for all $\be\in S$ and all $\bE$ 
satisfying $|E_j|=e_j$, with $c$ independent of $\be,\bE$. 

Let $\lambda\in\reals^+$ be a large positive constant, to be chosen below.
Like the constant $c$ in \eqref{eq:onesided}, $\lambda$
will depend on the compact set $S$ to which $\be$ is confined, but not otherwise on $\bE$. 
It is shown in \cite{christRSult} that
there exist sets $E_j^\dagger$ such that $\Estar_j\symdif E_j^\dagger
\subset \Estar_j\symdif E_j$, $|E_j^\dagger|=|E_j|$,
$E_j^\dagger\symdif \Estar_j\subset\{x: \big|\,|x|-e_j/2\,\big|\le\lambda \delta\}$,
and 
\[|\{x\in E_j\symdif E_j^\dagger: \big|\,|x|- e_j/2\,\big| 
\ge\lambda\delta\}| \ge \tfrac12 |E_j\symdif E_j^\dagger|.\]
Define $\bEdagger = (E_j^\dagger: j\in J)$,
and $f_j^\dagger = \one_{E_j^\dagger}-\one_{\Estar_j}$.

\begin{lemma} \label{lemma:daggercomparison}
There exists a constant $\lambda<\infty$ such that
for all $\be\in S$, all $\bE$ satisfying $|E_j|=e_j$
for all $j\in J$, and all sufficiently small $\delta>0$,
\begin{equation} \label{eq:daggercomparison}
\scriptt(\bE) \le \scriptt(\bEdagger) - c\lambda \delta \sum_j |E_j\symdif E_j^\dagger|.
\end{equation}
where $C,c\in\reals^+$ are independent of $\lambda,\delta$
so long as $\lambda\delta\le 1$. 
\end{lemma}

Here, and below, $\delta$ denotes $\dorbit$, as in \eqref{eq:dorbit}.
Since we aim to establish a conclusion only for
all sufficiently small $\delta$, there is no loss of generality
in requiring that $\lambda\delta\le 1$. 

\begin{proof}
In the integral defining $\Phi(\bE)$, substitute $\one_{E_j}=\one_{\Estar_j}+f_j$
and exploit multilinearity to expand the integral as a sum of $2^{|J|}$ terms. 
Those terms in which one single $f_j$ appears are of the form $\langle K_j,f_j\rangle$.
Provided that $\lambda\delta\le 1$,
\begin{equation}
\langle K_j,f_j\rangle \le \langle K_j,f_j^\dagger\rangle 
- c\lambda\delta |E_j\symdif E_j^\dagger| \ \text{ for each $j\in J$.}
\end{equation}

Consider each term in which $f_j$ appears for two or more distinct indices.
For each such index, express $f_j = f_j^\dagger + g_j$
where $g_j=f_j-f_j^\dagger$, and again use the multilinearity
of $\Phi$ to expand into finitely many terms, involving $\one_{\Estar_i},f_j^\dagger$,
and $g_k$ for various indices $i,j,k$.
By collecting all the resulting terms that do not involve any $g_j$,
and adding their sum to $\Phi(\bEstar) + \sum_{j\in J} \langle K_j,f_j^\dagger\rangle$,
we obtain the full expansion for $\Phi(\bEdagger) = \Phi(\one_{\Estar_j} + f_j^\dagger: j\in J)$.

Every term that remains 
involves one or more $f_i^\dagger$ and one $g_j$, or involves two or more $g_k$.
Since no two $L_j$ are colinear,
and since $\norm{f_k^\dagger}_\infty,\norm{g_k}_\infty\le 1$, $\norm{f_i^\dagger}_1\le 2\delta$,
$\norm{g_i}_1\le |E_i\symdif E_i^\dagger|$, and $|E_i\symdif E_i^\dagger|\le 2\delta$,
the contribution of any such term is $O(\delta \max_j |E_j^\dagger\symdif\bEstar|)$
by Lemma~\ref{lemma:addon}.
We conclude that
\begin{equation}
\scriptt(\bE) \le \scriptt(\bEdagger) 
- c\lambda\delta \sum_j |E_j\symdif E_j^\dagger|
+ C\delta \sum_j |E_j\symdif E_j^\dagger|.
\end{equation}
The conclusion \eqref{eq:daggercomparison} of the lemma follows if $\lambda$ is sufficiently large.
\end{proof}

Choose and fix a constant $\lambda$ sufficiently large for Lemma~\ref{lemma:daggercomparison}
to apply. Assume henceforth that $\delta\le\lambda^{-1}$.
Two useful conclusions can be drawn. 
Firstly,  
$\scriptt(\bE)\le\scriptt(\bEdagger)$.
Secondly, since $\scriptt(\bEdagger) \le \scriptt(\bEstar)$ by the Brascamp-Lieb-Luttinger inequality,
if $\max_j |E_j\symdif E_j^\dagger| > \tfrac14 \delta = \tfrac14\dist(\bE,\scripto(\bEstar))$
then $\scriptt(\bE)\le \scriptt(\bEstar)-c\delta^2$, as was to be shown.

There remains the case in which
$\max_j |E_j\symdif E_j^\dagger|\le \tfrac14 \delta = \tfrac14\dist(\bE,\scripto(\bEstar))$.
In this case,
\[\dist(\bEdagger,\scripto(\bEstar)) \ge \tfrac12\dist(\bE,\scripto(\bEstar)).\]
If we can show that 
\[\scriptt(\bEdagger)\le \scriptt(\bEstar)-c\dist(\bEdagger,\scripto(\bEstar))^2\]
then the proof of Theorem~\ref{thm:main2} will be complete. 
Equivalently, we have reduced matters to the case in which $\bE$ has the
supplementary property that $E_j\symdif \Estar_j\subset
\{x\in\reals: \big|\,|x|-e_j/2\,\big|\le C_0\dist(\bE,\scripto(\bEstar))\}$ for every $j\in J$,
where $C_0$ is some finite constant. This constant is not at our disposal to be chosen below; it
is proportional to the chosen $\lambda$, so must be regarded as given.
The notation $\bEdagger$ will not be used below. Instead, 
we analyze tuples $\bE$ that possess this supplementary property,
and  denote $\dist(\bE,\scripto(\bEstar))$ by $\delta$.

\subsection{Perturbation analysis: second order expansion} 

For any two distinct indices $i,j\in J$
define $L_{i,j}:\reals^2\to[0,\infty)$ by 
\begin{equation} \iint L_{i,j}(x,y)f_i(x)f_j(y)\,dx\,dy
= \Phi(\bg), \end{equation}
where $\bg=(g_n: n\in J)$ is defined by $g_i=f_i$, $g_j=f_j$,
and $g_k = \one_{E_k^\star}$ for every $k\notin\{i,j\}$.
We write $\langle L_{i,j},f_i\otimes f_j\rangle = \iint L_{i,j}(x,y)f_i(x)f_j(y)\,dx\,dy$.

\begin{lemma} \label{lemma:LijLipschitz}
Let $\be\in S$.
For each $i\ne j\in J$ there exists a neighborhood of each of the $4$ points
$(\pm e_i/2,\pm e_j/2)$, in which $L_{i,j}$ is Lipschitz continuous. 
\end{lemma}

\begin{proof}
Suppose without loss of generality that both signs are $+$;
one can reduce to this case by replacing $L_i$ and/or $L_j$
by $-L_i,-L_j$, respectively.
If $m=2$, then $L_{i,j}$ is constant in some neighborhood of $(e_i/2,e_j/2)$; 
this constant is the reciprocal of the absolute value of the Jacobian determinant
of the mapping $\reals^m\to\reals^2$ defined by $\bx\mapsto (L_i(x),L_j(x))$.

Suppose that $m\ge 3$.
According to the Brunn-Minkowski inequality,
$\log L_{i,j}$ is a concave function on $\{(x,y)\in\reals^2: L_{i,j}(x,y)\ne 0\}$,
so if $L_{i,j}(e_i/2,e_j/2)\ne 0$ then $L_{i,j}$ is Lipschitz in a neighborhood of $(e_i/2,e_j/2)$. 
We claim that $L_{i,j}$ either
vanishes identically in some neighborhood of $(e_i/2,e_j/2)$,
in which case it is certainly locally Lipschitz, or $L_{i,j}(e_i/2,e_j/2)\ne 0$. 
In either case, the proof would be complete.

Suppose to the contrary that $L_{i,j}(e_i/2,e_j/2)=0$,
but that $L_{i,j}$ does not vanish identically in any neighborhood of this point.
Then for any $\eps>0$ there exists $\by\in\scriptk_\be$
that satisfies $L_k(\by)\ge e_k/2 -\eps$ for $k=i$ and for $k=j$.
Indeed, 
if $L_{i,j}(y,y')\ne 0$ then there exist functions $g_i,g_j$ supported in arbitrarily
small neighborhoods of $y,y'$ respectively such that
$\int_{\reals^m} \prod_{k\in J} g_k\circ L_k\ne 0$.
where $g_k=\one_{E_k^\star}$ for each $k\notin\{i,j\}$.
Thus there exists $\bx\in\scriptk_\be$ with $L_i(\bx),L_j(\bx)$ equal to $y,y'$, respectively. 

Since $\scriptk_\be$ is compact, 
there must consequently exist $\bz\in\scriptk_\be$ that satisfies $L_k(\bz)=e_k/2$ for $k=i,j$.
Therefore there exists an extreme point $\bx$ of $\scriptk_\be$ that also satisfies these two equations.

Define $\scriptk_{\be,i,j}=\{\bx\in\reals^m: |L_n(x)|\le e_n/2
\text{ for every } n\in J\setminus\{i,j\} \}$.
The point $\bx$ belongs to $\scriptk_{\be,i,j}$. 
For any $s,t\in\reals$, $L_{i,j}(s,t)$ is equal to a nonzero constant multiple
of the $m-2$--dimensional Lebesgue measure of the set of all $\bz\in\scriptk_{\be,i,j}$
that satisfy $L_i(\bz)=s$ and $L_j(\bz)=t$.

According to the genericity hypothesis,
there exists a neighborhood of $\bx$ in which $\scriptk_\be$
is defined by $m$ linearly independent inequalities
$L_n(\by)\le e_n/2$ or $L_n(\by)\ge -e_n2$ for $n\in J_\bx$; both $i$ and $j$
are among the $m$ elements of $J_\bx$.
By replacing $L_n$ by $-L_n$ as necessary, we may without loss of generality arrange the signs
so that each of these inequalities is of the form $L_n(\by)\le e_n/2$. 
Moreover, $J_n(\bx)$ is strictly less than $e_n/2$ for every $n\in J\setminus J_\bx$.
Therefore in a neighborhood of $\bx$,
the boundary of $\scriptk_{\be,i,j}$ coincides with the $2$--dimensional affine subspace 
defined by $m-2$ equations
$L_n(\by) = \pm e_n/2$, with $n$ varying over  $J_\bx\setminus\{i,j\}$.
The mapping from this affine subspace to $\reals^2$ defined by $\by\mapsto h(\by) = (L_i(\by),L_j(\by))$
is an affine bijection. 
Thus $(e_i/2,e_j/2)$ lies in the interior of the image in $\reals^2$
of $\scriptk_{\be,i,j}$ under $h$. Since $\scriptk_\be$ is convex and has nonempty
interior, for each $z$ in the interior of $h(\scriptk_{\be,i,j})$,
$h^{-1}(z)$ contains a nonempty subset that is open in the relative topology,
hence has strictly positive $m-2$--dimensional Lebesgue measure.
In particular, $L_{i,j}(e_i/2,e_j/2)>0$.
\end{proof}

\begin{lemma} \label{lemma:quadraticexpansion}
Let $\lambda<\infty$.
Under the hypotheses of Theorem~\ref{thm:main},
there exists $\delta_0>0$ such that if 
$E_j\symdif \Estar_j\subset\{x: \big|\,|x|-e_j/2\,\big|\le\lambda\delta_0\}$ for each $j\in J$ then
\begin{equation} \scriptt(\bE) = \scriptt(\bEstar) + \sum_{j\in J} 
\langle K_j,f_j\rangle + \sum_{i<j\in J} \langle L_{i,j},f_i\otimes f_j\rangle + O(\dorbit^3),
\end{equation}
with the second summation taken over all distinct indices in $J$.
\end{lemma}

Before proving Lemma~\ref{lemma:quadraticexpansion}, we record a consequence.

\begin{corollary} \label{cor:quadexpansionsmall}
Under the hypotheses of Lemma~\ref{lemma:quadraticexpansion},
\begin{equation} \label{quadraticUB}
\sum_{j\in J} \langle K_j,f_j\rangle
+ \sum_{i < j\in J} \langle L_{i,j},f_i\otimes f_j\rangle \le O(\dorbit^3).
\end{equation}
\end{corollary}

This provides an upper bound only for the left-hand side of \eqref{quadraticUB}, 
not for its absolute value.  It follows from 
Lemma~\ref{lemma:quadraticexpansion} together with the 
Brascamp-Lieb-Luttinger inequality $\Phi(\bE)-\Phi(\bEstar)\le 0$.
\qed

\begin{proof}[Proof of Lemma~\ref{lemma:quadraticexpansion}]
Consider first the case in which $m=2$. 
If $\delta_0$ is sufficiently small then
\[ \Phi(\bE) = \Phi(\bEstar) + \sum_{k\in J} \langle K_k,f_k\rangle
+ \sum_{i<j} \langle L_{i,j},f_i\otimes f_j\rangle.\]
Indeed, if $i,j,k\in J$ are three distinct indices
then by the strict admissibility hypothesis,
there exists no $\bx\in\scriptk_\be$ satisfying $|L_n(\bx)| = e_n/2$
for all three values $n\in\{i,j,k\}$. 
From the compactness of $S$ it follows that there exists $\eta>0$
such that for any $\be\in S$ there exists no $\bx\in\reals^m$
satisfying $|L_n(\bx)-e_n/2|\le \eta$ for each $n\in\{i,j,k\}$
and $|L_m(\bx)|\le \tfrac12 e_m + \eta$ for every $m\in J\setminus\{i,j,k\}$.
Thus  provided that $\delta$ is sufficiently small, the threefold product
$f_i(L_i(\bx))f_j(L_j(\bx))f_k(L_k(\bx))$ vanishes identically as a function of $\bx\in\scriptk_\be$.
Therefore
\[f_i(L_i(\bx))f_j(L_j(\bx))f_k(L_k(\bx)) \prod_{m\in J\setminus\{i,j,k\}} g_m(L_m(\bx))\equiv 0
\text{ on } \reals^m\]
whenever each $g_m$ equals either $f_m$ or $\one_{E_m^\star}$.
Therefore when $m=2$, each term in the expansion of $\Phi(\bE)$ in which
at least three factors $\one_{E_n^\star}\circ L_n$
are replaced by $f_n\circ L_n$, vanishes. 
The $O(\delta^3)$ remainder term is in fact equal to zero in this case.

Next, suppose that $m\ge 3$. 
Again, consider any term in the expansion of $\Phi(\bE)$
in which $\one_{E_j^\star}$ is replaced by $f_j$
for at least $3$ distinct indices $j$. Let $J'$ be the set of all such indices
for this particular term.
The integrand in the integral defining this term is equal to
\[ \prod_{j\in J'} f_j(L_j(\by))
 \prod_{i\in J\setminus J'} \one_{E_i^\star}(L_i(\by)).\]
If there exists no point $\by\in\scriptk_\be$
satisfying $|L_j(\by)|=e_j/2$ for every $j\in J'$
then the integrand vanishes identically, provided that $\delta_0$ is sufficiently small, 
as in the discussion for $m=2$, above.

If there does exist $\by\in\scriptk_\be$ satisfying $|L_j(\by)|=e_j/2$ for every $j\in J'$ 
then there exists an extreme point $\by'$ of $\scriptk_\be$ that satisfies the same set of equations.
Then according to Lemma~\ref{lemma:independence}, $\{L_j: j\in J'\}$ is linearly independent.
In that case, for any $B<\infty$ there exists $C<\infty$ such that
whenever $|A_i|\le B$ for all $i\in J\setminus J'$,
\[ \int_{\reals^m}\prod_{j\in J'} \one_{A_j}(L_j(\bx))\prod_{i\in J\setminus J'}\one_{A_i}(L_i(\bx))\,d\bx
\le C\prod_{j\in J'}|A_j|.\]
Indeed, choose $\tilde J\subset J$
to contain $J'$ and to be a basis for $\reals^{m*}$.
Then
\[ \int_{\reals^m}\prod_{j\in J} \one_{A_j}(L_j(\bx))\,d\bx
\le \int_{\reals^m}\prod_{j\in \tilde J} \one_{A_j}(L_j(\bx))\,d\bx
= c\prod_{j\in \tilde J} |A_j|
\le cB^{|\tilde J|-|J'|}
\prod_{j\in \tilde J} |A_j| \]
where $c$ depends only on $\{L_n: n\in\tilde J\}$.
Since $|f_j|\le \one_{E_j\symdif \Estar_j}$, $|E_j\symdif \Estar_j|=O(\dorbit)$,
and $|J'|\ge 3$, $\prod_{j\in J'}|E_j\symdif \Estar_j| = O(\dorbit^3)$.
\end{proof}

\subsection{Perturbation analysis: exploitation of cancellation}

\begin{lemma} \label{lemma:nearlyinterval}
There exists $C<\infty$ such that for any $\scriptl,S,\be,\bE$
satisfying the above hypotheses, 
for each $n\in J$ there exists $w_n\in\reals$ such that
\begin{equation}
\int_{(E_n+w_n)\symdif\Estar_n} \big|\,|x|-\tfrac{e_n}2\,\big|\,dx  
\le C(\scriptt(\bEstar)-\scriptt(\bE)) + C\dist(\bE,\scripto(\bEstar))^3.
\end{equation}
\end{lemma}

Let $\delta=\dorbit$.
The proof of Lemma~\ref{lemma:nearlyinterval} exploits a reduction
of each quadratic form $\langle L_{i,j},f_i\otimes f_j\rangle$ to a corresponding quadratic
form on $\lt(S^0)$.
By the unit sphere $S^0\subset\reals$ we mean $\{-1,1\}$;
by $\int_{S^0}F$ we mean $F(1)+F(-1)$.
Define $F_j:S^0\to\reals$ by 
\begin{equation}\label{eq:Fjdefn} F_j(t) = \int_{|x-te_j/2|\le C\lambda\dorbit} f_j(x)\,dx
\ \text{ for $t=\pm 1$.}  \end{equation}
For $i\ne j\in J$ define $\ovl_{i,j}:S^0\times S^0\to \reals$
by \begin{equation} \ovl_{i,j}(s,t) = L_{i,j}(se_i/2,\,te_j/2).\end{equation}
Write
\[ \langle \ovl_{i,j},F_i\otimes F_j\rangle
= \iint_{S^0\times S^0} \ovl_{i,j}(s,t) F_i(s)F_j(t)\,ds\,dt. \]

\begin{lemma}\label{lemma:S0expansion}
Under the hypotheses of Theorem~\ref{thm:main2} and with the definitions and notations introduced above,
and assuming that $E_j\symdif \Estar_j\subset\{x: \big|\,|x|-e_j/2\,\big|\le C\lambda\dorbit\}$
for each $j\in J$,
\begin{equation} \label{eq:S0exp}
\langle L_{i,j},f_i\otimes f_j\rangle 
= \langle \ovl_{i,j},F_i\otimes F_j\rangle 
+ O(\dorbit^3).  
\end{equation}
\end{lemma}

\begin{proof}
Let $\delta=\dorbit$.
Each $f_j$ vanishes outside a $C\delta$--neighborhood of $\{\pm e_j/2\}$.
By Lemma~\ref{lemma:LijLipschitz}, $L_{i,j}$ is Lipschitz in some neighborhood
of each ordered pair $(\pm e_i/2,\pm e_j/2)$. 
Each point of the support of $f_i$ satisfies $|x-(\pm e_j/2)|=O(\delta)$,
and likewise for $f_j$. If $x_i,x_j$ are close to $e_i/2,e_j/2$, respectively,
then $|f_i(x_i)f_j(x_j)-F_i(1)F_j(1)|=O(\delta)$, with corresponding
bounds for the other points of $S^0\times S^0$.
Therefore
\begin{align*} \langle L_{i,j},f_i\otimes f_j\rangle 
&= \langle \ovl_{i,j},F_i\otimes F_j\rangle + O(\delta)\norm{f_i}_{L^1}\norm{f_j}_{L^1}  
\\&
= \langle \ovl_{i,j},F_i\otimes F_j\rangle + O(\delta^3).  \end{align*}
The conclusion of Lemma~\ref{lemma:S0expansion} now follows directly from
Lemma~\ref{lemma:quadraticexpansion}. 
\end{proof}

\begin{lemma} \label{lemma:balancing}
Let $k\in J$.
There exists $\by\in\reals^m$ such that $|\by|=O(\dorbit)$ and 
the function $\tilde F_k\in \lt(S^0)$ associated via \eqref{eq:Fjdefn} to
$\tilde E_k = E_k+L_k(\by)$ satisfies $\tilde F_k\equiv 0$. 
\end{lemma}

\begin{proof}
Let $\delta=\dorbit$.
Since $\tilde F_k(1)+\tilde F_k(-1) = \int_\reals \tilde f_k=0$,
it suffices to show that there exists $\by$ such that $\tilde F_k(1)=0$.
Choose $\bv\in\reals^m$ such that $L_k(\bv) > 0$.
Let $h(t)$ be $\tilde F_{k,t}(1)$, where $\tilde F_{k,t}$ is associated via \eqref{eq:Fjdefn} to $E_k+L_k(t\bv)$;
thus so long as $|t|$ is sufficiently small,
\begin{equation}
h(t) = |E_k\cap [\tfrac12 e_k-tL_k(\bv),\infty)| - |[\tfrac14 e_k,\tfrac12 e_k-tL_k(\bv)]\setminus E_k||
\end{equation}
since $E_k\symdif\Estar_k$ is contained in a $C\delta$--neighborhood of $\{\pm e_k/2\}$. 
This is a continuous function of $t$.
Let $B$ be a large constant, independent of $\delta$.
If $t= -B\delta$ then 
$E_k\cap [\tfrac12 e_k-tL_k(\bv),\infty)=\emptyset$,
so $h(t)\le 0$.
If $t= B\delta$ then 
$|E_k\cap [\tfrac12 e_k-tL_k(\bv),\infty)|>0$,
while  $[\tfrac14 e_k,\tfrac12 e_k-tL_k(\bv)]\setminus E_k|=\emptyset$, so $h(t)>0$.
Therefore by the Intermediate Value Theorem, there exists $t\in[-Bt,Bt]$ satisfying $h(t)=0$.
\end{proof}

\begin{proof}[Proof of Lemma~\ref{lemma:nearlyinterval}]
Let $n\in J$. By Lemma~\ref{lemma:balancing} together with the invariance of
$\Phi(\bE)$ under the symmetries $(E_j)\mapsto (E_j+L_j(\bv))$,
we may suppose without loss of generality that 
the associated function $F_n\in\lt(S^0)$ satisfies $F_n\equiv 0$.

Let $J'=J\setminus\{n\}$.
Since any second order term involving $F_n$ vanishes,
\eqref{eq:S0exp} simplifies to
\begin{equation}
\scriptt(\bE)
= \scriptt(\bEstar)
+\sum_{k\in J} \langle K_k,f_k\rangle
+ \sum_{i <j\in J'} 
\langle \ovl_{i,j},F_i\otimes F_j\rangle
+O(\delta^3).
\end{equation}
This expression is rather favorable, for the term $\langle K_n,f_n\rangle$
is rather negative unless $E_n$ nearly coincides with an interval,
while there is no term $\langle \bar L_{i,j},\,F_i\otimes F_j\rangle$
with $i$ or $j$ equal to $n$ to potentially offset this negative term.

Define $\tilde E_j = E_j$ for all $j\ne n$,
and $\tilde E_n = E_n^\star$.
Define $\tilde f_j,\tilde F_j$ to be the associated functions.
Then $\tilde f_i = f_i$ and $\tilde F_i=F_i$ for $i\ne n$, 
while $\tilde f_n\equiv 0$ and $\tilde F_n\equiv 0$.
By Lemma~\ref{lemma:S0expansion},
\[ \sum_{k\in J'} \langle K_k,f_k\rangle
+ \sum_{i < j\in J'} \langle \ovl_{i,j},F_i\otimes F_j\rangle
= \sum_{k\in J'} \langle K_k,f_k\rangle + \sum_{i < j\in J'} 
\langle L_{i,j},f_i\otimes f_j\rangle +O(\delta^3).\] 
By the definitions of $\tilde f_i,\tilde F_i$,
\[
\sum_{k\in J'} \langle K_k,f_k\rangle + \sum_{i < j\in J'} 
\langle L_{i,j},f_i\otimes f_j\rangle
=
\sum_{k\in J} \langle K_k, \tilde f_k\rangle + \sum_{i < j\in J} 
\langle L_{i,j},\tilde f_i\otimes \tilde f_j\rangle.\] 
By applying Corollary~\ref{cor:quadexpansionsmall} to 
$(\tilde E_j: j\in J)$, we conclude that the right-hand side of this equation is $\le O(\delta^3)$. 
Thus we have shown that 
\begin{equation}
\sum_{k\in J'} \langle K_k,f_k\rangle
+ \sum_{i < j\in J'} \langle \ovl_{i,j},F_i\otimes F_j\rangle
\le O(\delta^3).
\end{equation}
Therefore according to \eqref{eq:onesided},
\begin{equation}
\scriptt(\bE) \le \scriptt(\bEstar)
-  c \int_{E_n\symdif\Estar_n} \big|\,|x|-\tfrac{e_n}2\,\big|\,dx +O(\delta^3),
\end{equation}
which is the desired conclusion.
\end{proof}


\begin{proof}[Proof of Proposition~\ref{prop:nearlyintervals}]
Let $w_j$ satisfy the conclusion of Lemma~\ref{lemma:nearlyinterval}.
Choose $I_j$ to be the interval centered at $-w_j$ satisfying $|I_j|=|E_j|$.  Then
\begin{equation*} |E_j\symdif I_j|^2 
= |(E_j+w_j)\symdif \Estar_j|^2
\le C \int_{(E_j+w_j)\symdif\Estar_j} \big|\,|x|-\tfrac{e_j}2\,\big|\,dx.
\end{equation*}
Thus Proposition~\ref{prop:nearlyintervals} follows from Lemma~\ref{lemma:nearlyinterval}.
\end{proof}

\section{Hybrid analysis}\label{section:hybrid}

The proof of Theorem~\ref{thm:main2} can now be completed by combining the quantitative forms 
of two facts established above for tuples $\bE$ that nearly maximize $\Phi$:
each set $E_j$ is nearly an interval, and if every $E_j$ is equal to an interval then the centers 
of those intervals must be nearly compatibly situated.

Define $\bard\le\delta$ by
\begin{equation} \bard= \max_j \inf_I |E_j\symdif I|, \end{equation}
where the infimum is taken over all intervals $I\subset\reals$ satisfying $|I|=|E_j|$.
For each index $j$ choose an interval
$I_j$ satisfying $|E_j\symdif I_j| \le 2\inf_I|E_j\symdif I|\le 2\bard$ and $|I_j|=|E_j|$.
Define
\begin{equation}
\tdelt = \inf_{\bv\in\reals^m} \max_j |I_j\symdif (\Estar_j+L_j(v_j))|.
\end{equation}
Then $\bard+\tdelt \asymp \delta$, that is, the ratio of $\bard+\tdelt$ to $\delta$
is bounded above and below by positive constants so long as $\delta$ is sufficiently small.
In this notation, the conclusion of Proposition~\ref{prop:nearlyintervals} can be restated as
\begin{equation}\label{eq:propNIrestated} \Phi(\bE)\le\Phi(\bEstar) -c\bard^2 + O(\delta^3).\end{equation} 

Choose $\bv\in\reals^m$ to satisfy $\max_j |I_j\symdif (\Estar_j+L_j(v_j))|\le 2\tdelt$.
Replace $E_j$ by $E_j-L_j(v_j)$,
and thus replace $I_j$ by $I_j-L_j(v_j)$, for all $j\in J$.
Thus $\bE$ is modified, but $\Phi(\bE)$ and $\dist(\bE,\scripto(\bEstar))$ are unchanged.

Let $A<\infty$ be a large constant, to be chosen below.
If $\bard\ge A^{-1}\delta$, then the desired inequality $\Phi(\bE)\le\Phi(\bEstar)-c\delta^2$
follows immediately from \eqref{eq:propNIrestated} for all sufficiently small $\delta$, 
with a smaller value of $c$ which depends on the choice of $A$ but is positive for any $A$.
Therefore we may, and do, assume henceforth that $\bard\le A^{-1}\tdelt$.

Set $\bI = (I_j: j\in J)$.
According to Proposition~\ref{prop:intervals},
\begin{equation} \Phi(\bI)\le \Phi(\bEstar) -c\tdelt^2. \end{equation} 
We will relate $\Phi(\bE)$ to $\Phi(\bI)$ in order to exploit this information.
Writing $\one_{E_j} = \one_{I_j}+f_j$, one has $\norm{f_j}_{L^1} = |E_j\symdif I_j|\le 2\bard$. 
Expand
\begin{equation} \label{eq:expanded}
\Phi(\bE) = \Phi(\bI) + \sum_j \langle K_{j,\bI},f_j\rangle + O(\bard^2)
\end{equation}
where $K_{j,\bI}$ are defined by 
\begin{equation}
\int_\reals f_j K_{j,\bI} = \int_{\reals^m} (f_j\circ L_j) \prod_{i\ne j} \one_{I_i}(L_i).
\end{equation}

The properties of the quantities $K_{j,\bI}$ in this expansion
are less favorable, in general, than those
of $K_j = K_{j,\bEstar}$. Nonetheless, we will show that 
\begin{equation} \label{eq:lastpoint} \langle K_{k,\bI},f_k\rangle \le O(\bard\tdelt)
\ \text{ for every $k\in J$.} \end{equation}
As in the analysis above, this is an upper bound merely
for the quantity on the left-hand side of the inequality, not for its absolute value.

Once \eqref{eq:lastpoint} has been established, it will follow using \eqref{eq:expanded} that
\begin{equation} \Phi(\bE) 
\le \Phi(\bEstar) -c\tdelt^2 + C\bard\tdelt + O(\bard^2)
\le \Phi(\bEstar) -c\tdelt^2 + CA^{-1} \tdelt^2.  \end{equation}
Choosing $A = 2Cc^{-1}$, this will complete the proof,
since $\tilde \delta$ is comparable to $\delta$ in the present case.

To establish \eqref{eq:lastpoint}, consider any $k\in J$. By 
replacing $\bE$ by $(E_j-L_j(\bv): j\in J)$ 
where $\bv\in\reals^m$ is chosen so that $L_k(\bv)$ is equal
to the center of $I_k$ and $|\bv|=O(\tdelt)$,
we may reduce to the case in which $I_k = \Estar_k$.  

\begin{lemma}\label{lemma:KkI}
\begin{equation}\label{eq:KkI} \norm{K_{k,\bI}-K_{k,\bEstar}}_{L^\infty} \le C\tdelt.  \end{equation}
\end{lemma}

\begin{proof}
Let $\varphi:\reals\to[0,\infty)$ be arbitrary. For $j\in J\setminus \{k\}$
define $g_j = \one_{E_j^\star}$, and define $h_j = \one_{I_j}$.
Then
\[ \int_\reals |K_k(\bI)-K_k(\bEstar)|\varphi
= \int_{\reals^m} \big| \prod_{j\ne k} \one_{I_j}\circ L_j-\prod_{j\ne k} \one_{E_j^\star}\circ L_j\big| \varphi\circ L_k.  \]
Choose $J'\subset J$ so that $k\in J'$
and $\{L_j: j\in J'\}$ is a basis for $\reals^{m*}$.
Then
\[ \int_\reals |K_k(\bI)-K_k(\bEstar)|\varphi_k
\le 2^{|J|-m}\, \norm{\varphi}_1\,\prod_{j\in J'\setminus\{k\}} (|I_j|+|E_j^\star|).  \]
\end{proof}

Recall that $\langle K_{k,\bEstar},f_k\rangle\le 0$, as was shown above.
\eqref{eq:lastpoint} follows directly from Lemma~\ref{lemma:KkI}: 
\begin{align*}
\langle K_{k,\bI},f_k\rangle
& = 
\langle K_{k,\bEstar},f_k\rangle
+ \langle K_{k,\bI}-K_{k,\bEstar},f_k\rangle
\\& \le 0 + \norm{K_{k,\bI}-K_{k,\bEstar}}_\infty\norm{f_k}_1
\\& \le C\tilde\delta\cdot 2\bard.
\end{align*}
This completes the proof of Theorem~\ref{thm:main2}.
\qed

\end{document}